\documentclass[hidelinks,onefignum,onetabnum]{siamart250211}
\pdfoutput=1

\usepackage{rotating}
\usepackage{subfig}
\usepackage{amsmath}
\usepackage[english]{babel}
\usepackage{setspace}
\usepackage{amsfonts}
\usepackage{graphicx}
\usepackage{epstopdf}
\usepackage{algorithmic}
\usepackage{mathtools}
\usepackage{multirow}
\usepackage{booktabs} 
\usepackage{caption}
\usepackage{makecell}
\usepackage{cleveref}
\usepackage{amsopn}
\usepackage{bm}  
\usepackage{amssymb}
\allowdisplaybreaks[4]

\newsiamremark{remark}{Remark}
\newsiamremark{hypothesis}{Hypothesis}
\crefname{hypothesis}{Hypothesis}{Hypotheses}
\newsiamthm{claim}{Claim}

\headers{}{}

\title{Fractional calculus via variable-transform-based spectral approximations}

\author{Xiaolin Liu\thanks{School of Mathematical Sciences, University of Science and Technology of China, 96 Jinzhai Road, Hefei 230026, Anhui, China (\email{xiaolin9907@mail.ustc.edu.cn},  \email{kuanxu@ustc.edu.cn}).}  \and Kuan Xu\footnotemark[1]}

\def\mA{\mathcal{A}}
\def\mB{\mathcal{B}}

\def\mD{\mathcal{D}}

\def\mF{\mathcal{F}}

\def\mI{\mathcal{I}}
\def\mJ{\mathcal{J}}

\def\mM{\mathcal{M}}

\def\mO{\mathcal{O}}

\def\mR{\mathcal{R}}
\def\mS{\mathcal{S}}

\def\md{\mathrm{d}}

\def\bf{\boldsymbol{f}}

\def\veps{\varepsilon}

\def\ema{\epsilon_{\mathrm{m}}}

\DeclareMathOperator{\arcsinh}{arcsinh}

\DeclareMathOperator{\arctanh}{arctanh}
\graphicspath{{./}}

\makeatletter
\renewcommand\paragraph{\@startsection{paragraph}{4}{1.5em}
  {1ex \@plus1ex \@minus.1ex}
  {-1em}
  {\normalfont\normalsize\bfseries}}
\makeatother

\begin{document}
\maketitle

\begin{abstract}
We present a novel and unifying framework for constructing spectral approximations to fractional integral operators. These spectral approximations are based on transplanted Chebyshev polynomials, which are obtained by composing Chebyshev polynomials with a variable transform. When an algebraic transform is used, the framework produces spectral approximations based on Jacobi fractional polynomials. When an exponential transform is used, it yields a versatile spectral approximation that is applicable to a much broader class of fractional calculus problems. The construction of such spectral approximations is both numerically stable and optimal in terms of complexity. These spectral approximations lead to stable and fast spectral methods for fractional calculus. The spectral method based on the double-exponential transform is demonstrated through extensive numerical examples that are intractable for existing methods.
\end{abstract}



\section{Introduction}\label{sec:intro}
In this paper, we propose a new framework for the design and construction of spectral approximations to the left- and right-sided fractional integral operators (FIOs)
\begin{subequations}
\begin{align}
_{-1}\mathcal{I}_{x}^{\mu}[u](x) = \frac{1}{\Gamma(\mu)}\int_{-1}^{x} \frac{u(t)}{(x-t)^{1-\mu}} \, \mathrm{d} t \label{fiol}
\end{align}
and
\begin{align}
_{x}\mathcal{I}_{1}^{\mu}[u](x) = \frac{1}{\Gamma(\mu)}\int_{x}^{1} \frac{u(t)}{(t-x)^{1-\mu}} \, \mathrm{d} t \label{fior}
\end{align}
\end{subequations}
for $x \in [-1,1]$, where $\Gamma(\cdot)$ denotes the Gamma function. By employing suitable variable transforms, the proposed framework yields spectral approximations to FIOs that are applicable to any fractional order $\mu > 0$. These approximations, in turn, give rise to highly versatile spectral methods capable of solving a significantly broader class of fractional calculus problems than those tractable by existing spectral methods.

Existing spectral approximations to FIOs or fractional differential operators (FDOs) suffer from several limitations, particularly when applied to fractional integral equations (FIEs) and fractional differential equations (FDEs). The collocation method based on polyfractonomials (PFC) \cite{zay}, the Petrov--Galerkin method using generalized Jacobi functions (GJFPG) \cite{che2}, and the sum-space method (SS) \cite{hal} all employ variants of weighted Jacobi polynomials as (part of) their basis functions. In particular, the PFC and GJFPG methods approximate algebraic singularities in the solution using polynomials or fractional-power polynomials that do not match the true singular behavior of the solution. Consequently, despite being termed ``spectral'', these methods often fail to achieve spectral convergence; instead, it is common for the computed solutions of FIEs or FDEs to converge only algebraically \cite{che2,pu,zay}. The SS method uses bases that span the direct sum of suitably weighted ultraspherical and Jacobi polynomial spaces, leading to matrix approximations of FIOs and FDOs that may be ill-conditioned \cite{pu}. Moreover, the SS method is restricted to FIEs and FDEs of rational order, i.e., $\mu = p/q$, where $p$ and $q$ are positive integers. It requires $2q$ distinct weighted orthogonal polynomial bases, and thus becomes increasingly cumbersome to implement when $q$ is large. From the perspective of computational efficiency, the matrix approximations constructed by the PFC method may be dense. Furthermore, the PFC, GJFPG, and SS methods all rely on Jacobi polynomials. For Jacobi polynomials with general integer or fractional parameters, the transforms between values and coefficients are significantly less efficient than those for Chebyshev or Legendre polynomials\footnote{Even for Legendre polynomials, existing algorithms for transforms between values and coefficients are not sufficiently fast. These algorithms are often referred to as ``galactic'' due to the high powers of $\log n$ in their asymptotic complexity and the large hidden constants in the big-$O$ notation.}.

The generalized log-orthogonal function Petrov--Galerkin (GLOFPG) method \cite{che} represents the solution to FDEs using generalized log-orthogonal functions (GLOFs)
\begin{align}
H^{(\delta, \nu, \rho)}_n(x) = x^{\frac{\nu-\rho}{2}}L^{(\delta)}_n\bigl(-(\nu+1)\log x\bigr), \label{glof}
\end{align}
and can, in principle, solve a broader class of problems than other existing methods. Despite this generality, several drawbacks hinder its practical applicability. The three most significant ones are as follows: (1) the condition number of the matrix approximations to FIOs or FDOs grows rapidly with the discretization size; even for approximation matrices of only a few hundred in dimension, the resulting linear systems can become so ill-conditioned that the computed solutions may contain virtually no accurate digits; (2) the evaluation of high-degree GLOFs is prone to numerical underflow or overflow, which further restricts the maximum feasible size of the matrix approximations; (3) for discretization size $N$, the overall $\mathcal{O}(N^4)$ computational complexity further limits the method in solving large-scale problems.

The recently developed JFP spectral method constructs spectral approximations to FIOs using a Chebyshev version of the Jacobi fractional polynomials (JFPs)
\begin{align}
Q_n^{\alpha,\beta}(x) = \left(\frac{1+x}{2}\right)^{\alpha} T_n\left(2\left(\frac{1+x}{2}\right)^{\beta}-1\right), \label{jfp}
\end{align}
which form an orthogonal basis with respect to the weight function $w(x) = (1+x)^{\frac{\beta}{2}-1 - 2\alpha} \left(1 - 2^{-\beta}(1+x)^{\beta}\right)^{-\frac{1}{2}}$. Here and throughout this paper, $T_n$ denotes the $n$th Chebyshev polynomial of the first kind. The parameters $\alpha > -1$ and $\beta > 0$ are determined by the fractional order $\mu$ and other properties of the given problem. These approximations can be constructed stably with optimal computational complexity and yield a well-conditioned spectral method for solving FIEs and FDEs. It has been shown that the JFP-based spectral method significantly outperforms all the aforementioned methods when it is applicable. For instance, the JFP spectral method can solve FIEs of the general form
\begin{align}
a_0(x) u(x) + a_1(x) \mathcal{I}^{\mu_1}\left[b_1(\rotatebox{45}{\scalebox{0.55}{$\square$}}) u\right](x) + \cdots + a_{\ell}(x) \mathcal{I}^{\mu_{\ell}}\left[b_{\ell}(\rotatebox{45}{\scalebox{0.55}{$\square$}}) u\right](x) = f(x) \label{fie}
\end{align}
under the following assumptions.

\begin{itemize}
\item The fractional orders $\mu_l > 0$ and $\mu_l = (p_l/q_l)\mu_1$ for $l = 2, \ldots, \ell$, where $p_l, q_l \in \mathbb{N}^+$ and $p_l/q_l$ is irreducible.
\newcounter{tempcounter}
\setcounter{tempcounter}{\value{enumi}}
\end{itemize}

Let $q$ denote the least common multiple of $\{q_l\}_{l=2}^{\ell}$, define $\tilde{\beta} = \mu_1/q$, and set $\mu_0 = 0$ and $b_0(x) = 1$.

\begin{itemize}
\setcounter{enumi}{\value{tempcounter}}
\item The variable coefficients $a_l(x)$ and $b_l(x)$ for $l = 0, 1, \ldots, \ell$, as well as the right-hand side $f(x)$, can all be approximated by finite expansions in $\{Q_n^{\alpha^{a_l},\tilde{\beta}}(x)\}_{n = 0}^{\infty}$, $\{Q_n^{\alpha^{b_l},\tilde{\beta}}(x)\}_{n = 0}^{\infty}$, and $\{Q_n^{\alpha^{f},\tilde{\beta}}(x)\}_{n = 0}^{\infty}$ for some $\alpha^{a_l}$, $\alpha^{b_l}$, and $\alpha^{f}$, respectively.

\item For $0 \leq j \neq l \leq \ell$, there exist $k_j, k_l \in \mathbb{N}$ such that  
\begin{align*}
&\alpha^{f}- \alpha^{a_l} - \mu_l - k_l\tilde{\beta} > -1,\\
&\bigl(\alpha^{a_j}-\alpha^{a_l}\bigr) + \bigl(\alpha^{b_j}-\alpha^{b_l}\bigr) + \bigl(\mu_j-\mu_l\bigr) + (k_j-k_l)\tilde{\beta} = 0.
\end{align*}
\end{itemize}

On the one hand, despite the complexity of these conditions, \cref{fie} under these additional assumptions covers a large class of FIEs encountered in practice, demonstrating the applicability of the JFP spectral method. On the other hand, we wish to be able to solve FIEs beyond this category, or more generally, fractional calculus problems in an even broader sense; see the examples in \cref{sec:example}.

In fact, the limitation of the JFP spectral method can be attributed to the algebraic transform whose inverse is given by
\begin{align}
y = \psi^{-1}(x) = 2\left(\frac{1+x}{2}\right)^{\beta}-1. \label{algebraic}
\end{align}
This inverse transform maps $[-1,1]$ onto itself and, more importantly, incorporates a singularity of order $\beta$ into the Chebyshev polynomials, so that the JFP basis can be regarded as polynomials in the fractional power $\left((1+x)/2\right)^{\beta}$. This transform is key to the success of the JFP method, but at the same time limits its applicability to more general problems.

The success and limitations of the JFP method nevertheless motivate us to consider variable transforms beyond those of purely algebraic type. Hence, in this paper we work with a general variable transform
\begin{align}
x = \psi(y) \label{mapping}
\end{align}
as far as possible before turning to two important specific instantiations. The only requirement is that $\psi(y)$ is monotonically increasing and maps $[-1, 1]$ onto itself. We shall take the transplanted Chebyshev polynomials (TCPs)
\begin{align}
Q_n(x) = T_n\left(\psi^{-1}(x)\right), \quad x \in [-1, 1] \label{tcp}
\end{align}
as the basis functions, as they are orthogonal on $[-1, 1]$ with respect to the weight function
\begin{align*}
w(x) = \left. \left(\frac{\mathrm{d}}{\mathrm{d} x} \psi^{-1}(x)\right) \middle/ \sqrt{1-\left(\psi^{-1}(x)\right)^2}\right..
\end{align*}
With this general choice of basis functions, we obtain a novel and unifying framework that produces spectral approximations to FIOs which are broadly applicable in fractional calculus without compromising computational efficiency. To be specific, consider a given function $u(x)$ that can be approximated by a TCP series. That is, $u(x) = \mathbf{Q}c$, where the quasimatrix $\mathbf{Q} = \left[Q_0(x), Q_1(x), Q_2(x), \dots \right]$ and $c = \left(c_0,c_1,c_2,\ldots\right)^{\top}$ is the coefficient vector. We aim to construct an infinite-dimensional matrix $\mA$ such that
\begin{align*}
\mI^{\mu}[u](x) = \mathbf{Q} \mA c,
\end{align*}
where $\mI^{\mu}$ is either $_{-1}\mathcal{I}_{x}^{\mu}$ or $_{x}\mathcal{I}_{1}^{\mu}$. With a concrete choice of variable transform, such spectral approximations give rise to spectral methods for solving a wide range of fractional calculus problems.

Throughout this paper, we say that a matrix $A$ has bandwidths $(\xi_l, \xi_u)$ if its entries satisfy $A_{ij}=0$ whenever $i-j>\xi_l$ or $j-i>\xi_u$. The notations $\Re( \rotatebox{45}{\scalebox{0.5}{$\square$}} )$ and $\Im( \rotatebox{45}{\scalebox{0.5}{$\square$}} )$ denote the real and imaginary parts, respectively. We denote by $\ema$ the machine epsilon of floating-point arithmetic. Following standard convention, the two-parameter Mittag--Leffler function is denoted by $E_{\zeta, \eta }$.

The remainder of this paper is organized as follows. In \cref{sec:frac}, we establish a unifying framework that gives rise to spectral approximations of FIOs for general variable transforms. In \cref{sec:vt}, these approximations are instantiated for algebraic and exponential transforms. In \cref{sec:example}, we demonstrate the capability of the spectral method based on the double-exponential transform using problems of various types that are intractable for existing spectral methods. Finally, we conclude the paper with discussions and an outlook in \cref{sec:remarks,sec:out}, respectively.

\section{Construction of the spectral approximation}\label{sec:frac}
In this section, we establish a unifying framework for constructing spectral approximations to FIOs under the general variable transform \cref{mapping}. We begin with several identities satisfied by the Chebyshev polynomials of the first and second kinds, $T_n(x)$ and $U_n(x)$. Their proofs can be found in standard references, e.g., \cite{sze}.  In the rest of this paper, we assume $U_{-1}(x) = 0$.
\begin{lemma}
For $n \geq 2$,
\begin{subequations}
\begin{align}
U_n(x) &= 2T_n(x) + U_{n-2}(x), \label{utu} \\
\frac{\md}{\md x}T_n(x) &= nU_{n-1}(x), \label{dtu} \\
U_n(x) &= 2x U_{n-1}(x) - U_{n-2}(x). \label{uuu}
\end{align}
\end{subequations}
\end{lemma}

\subsection{Left-sided FIOs}\label{sec:left}
We consider the left-sided FIO \cref{fiol} by applying it to $u(x)$. By the linearity of the FIO, it follows that
\begin{align*}
_{-1}\mathcal{I}_{x}^{\mu}[u](x) = 
\left( 
    \vphantom{\begin{aligned}
      & \\ 
      & \\ 
      &
  \end{aligned}}
  \right.
  \begin{aligned}
    _{-1}\mathcal{I}_{x}^{\mu}[Q_0](x)\Bigg| ~_{-1}\mathcal{I}_{x}^{\mu}[Q_1](x) \Bigg| ~_{-1}\mathcal{I}_{x}^{\mu}[Q_2](x) \Bigg| \cdots  
  \end{aligned}
  \left.
  \vphantom{\begin{aligned}
      & \\ 
      & \\ 
      &
  \end{aligned}}
  \right)
c,
\end{align*}
and it therefore suffices to consider the columns individually. A key observation is that expanding $_{-1}\mathcal{I}_{x}^{\mu}[Q_n](x)$ in $\{Q_n(x)\}_{n=0}^{\infty}$ is equivalent to expanding $_{-1}\mathcal{I}_{x}^{\mu}[Q_n](\psi(y))$ in the Chebyshev polynomials $\{T_n(y)\}_{n=0}^{\infty}$. Hence, we apply the forward transform $x = \psi(y)$ and the change of variables $t = \psi(\tau)$ to $_{-1}\mathcal{I}_{x}^{\mu}[Q_n](x)$ to obtain
\begin{subequations}
\begin{align}
_{-1}\mathcal{I}_{x}^{\mu} &\left[Q_n\right](\psi(y)) = \frac{1}{\Gamma(\mu)} \int_{-1}^{y} \frac{T_n(\tau)\psi'(\tau)}{\left(\psi(y)-\psi(\tau)\right)^{1-\mu}} \md \tau \label{psi1} \\ 
&= \frac{1}{\Gamma(1+\mu)}\biggl((-1)^n\left(1+\psi(y)\right)^{\mu} + n\underbrace{\int_{-1}^{y} U_{n-1}(\tau)\left(\psi(y)-\psi(\tau)\right)^{\mu}\md \tau }_{\mathclap{\varphi_n(y)}} \biggr), \label{IQnl}
\end{align}
\end{subequations}
where the last equality follows from integration by parts and the identity $T_n(-1)=(-1)^n$. Let us denote by $\varphi_n(y)$ the integral term in \cref{IQnl} and apply a third change of variables
\begin{align}
\tau = y - (1+y)t \label{cov}
\end{align}
to $\varphi_n(y)$, where we reuse the notation $t$. It then becomes a definite integral:
\begin{subequations}
\begin{align}
\varphi_n(y) &= (1+y)\int_{0}^{1} U_{n-1}\bigl(y-(1+y)t\bigr)\bigl(\psi(y)-\psi(y-(1+y)t)\bigr)^{\mu}\md t \label{psi0} \\
&= (1+y)\int_{0}^{1} t^{\mu} U_{n-1}\bigl(y-(1+y)t\bigr)G(y,t)\md t, \label{psi}
\end{align}
\end{subequations}
where, for $(y,t)\in[-1,1]\times[0,1]$, the bivariate function $G(y, t)$ is defined by
\begin{align}
G(y,t) = \left(\frac{\psi(y)-\psi(y-(1+y)t)}{t}\right)^{\mu}. \label{G}
\end{align}
We now assume that $G(y,t)$ can be represented or approximated by a rank-$r$ approximant. That is,
\begin{align}
G(y,t) \approx \sum_{j=1}^r \sigma_j f_j(y) g_j(t), \label{Glr}
\end{align}
where $f_j(y)$ and $g_j(t)$ are either smooth or smooth apart from endpoint singularities. We defer the construction of \cref{Glr} to \cref{sec:vt} when $\psi(y)$ is instantiated. Substituting \cref{Glr} into \cref{psi}, we obtain
\begin{align}
\varphi_n(y) = \left(1+y\right) \sum_{j=1}^{r} \sigma_j f_j(y) \varphi_{n}^{j}(y), \label{psin}
\end{align}
where the moment is given by
\begin{align*}
\varphi_{n}^{j}(y) = \int_{0}^{1} t^{\mu} U_{n-1}\left(y - (1+y)t\right) g_j(t) \md t.
\end{align*}
It is readily seen that $\varphi_n^j(y)$ is a polynomial in $y$ of degree at most $n-1$. The following theorem is the main mathematical result of this paper, showing that $\varphi_n^j(y)$ satisfies a three-term recurrence relation for any admissible variable transform.
\begin{theorem}
For $n \geq 2$ and all $j$,
\begin{align}
\left((1+y)\frac{\md}{\md y} - n\right)\varphi_{n+1}^j(y) = \left((1+y)\frac{\md}{\md y} + n\right)\varphi_{n-1}^j(y) + 2n\varphi_n^j(y). \label{rec}
\end{align}
\end{theorem}
\begin{proof}
By \cref{utu}, we have
\begin{align}
\frac{\md}{\md y} \varphi_{n+1}^j(y)
&= \int_{0}^{1} t^{\mu}\frac{\md}{\md y}\bigl(2T_n\left(y - (1+y)t\right) + U_{n-2}\left(y - (1+y)t\right)\bigr) g_j(t)\md t \nonumber \\
&= 2n\int_{0}^{1} t^{\mu}(1-t)U_{n-1}\left(y - (1+y)t\right) g_j(t)\md t + \frac{\md}{\md y}\varphi_{n-1}^j(y), \label{rec1_int}
\end{align}
where \cref{dtu} is used in the last equality. Denote the integral in \cref{rec1_int} by $I(y)$. Then
\begin{align}
I(y)
&= \varphi_n^j(y) - \int_{0}^{1} t^{\mu+1}U_{n-1}\left(y - (1+y)t\right) g_j(t)\md t \nonumber \\
&= \varphi_n^j(y) + \frac{1}{1+y}\int_{0}^{1} t^{\mu}\left(y-(1+y)t-y\right)U_{n-1}\left(y - (1+y)t\right) g_j(t)\md t \nonumber \\
&= \frac{1}{1+y}\varphi_n^j(y) + \frac{1}{1+y}\int_{0}^{1} t^{\mu}\left(y-(1+y)t\right)U_{n-1}\left(y - (1+y)t\right) g_j(t)\md t \nonumber \\
&= \frac{1}{1+y}\varphi_n^j(y) + \frac{1}{2(1+y)}\left(\varphi_{n+1}^j(y) + \varphi_{n-1}^j(y)\right), \label{I}
\end{align}
where we have used \cref{uuu}. Substituting \cref{I} into \cref{rec1_int} and simplifying yields \cref{rec}.
\end{proof}

Since \cref{rec} is a first-order ODE, we require the initial values $\varphi^j_0(y)$, $\varphi^j_1(y)$, and $\varphi^j_2(y)$, together with a boundary condition to carry out the recurrence. The coefficient $(1+y)$ vanishes at $y=-1$, so the ODE \cref{rec} is singular at the left endpoint and does not satisfy the assumptions required by the Picard--Lindel\"{o}f theorem there. Therefore, we impose the right Dirichlet boundary condition, which requires the value of $\varphi_{n+1}^j(y)$ at $y=1$. That is, we need to evaluate
\begin{align}
\varphi_n^j(1) = \int_{0}^{1} t^{\mu} U_{n-1}\left(1 - 2t\right) g_j(t)\,\md t \label{psink1}
\end{align}
for $n \ge 3$. The computation of the initial values and the integral \cref{psink1} depends on the choice of the variable transform and is discussed in \cref{sec:vt}. For the moment, we assume that they are available.

Let $\mR^j$ be the infinite-dimensional matrix whose $n$th column $\mR^j_n$ collects the Chebyshev coefficients of $\varphi_n^j(y)$. We solve \cref{rec} using the ultraspherical spectral method \cite{olv}, which yields the infinite linear system
\begin{align}
\begin{pmatrix*}[c]
\mB \\[1mm]
\mM \mD - n\mS
\end{pmatrix*}
\mR_{n+1}^j
=
\begin{pmatrix*}[c]
\varphi_{n+1}^j(1) \\[1mm]
\mF
\end{pmatrix*}, \label{systemleft}
\end{align}
where
\begin{align*}
\mF = \left( \mM \mD + n\mS \right)\mR_{n-1}^j + 2n\mS \mR_n^j.
\end{align*}
Here, $\mD$, $\mM$, and $\mS$ denote the first-order differentiation matrix mapping Chebyshev polynomials of the first kind $\{T_n(y)\}_{n=0}^{\infty}$ to those of the second kind $\{U_n(y)\}_{n=0}^{\infty}$, the multiplication matrix representing multiplication by $1+y$, and the conversion matrix from Chebyshev coefficients of the first kind to those of the second kind, respectively. The resulting infinite-dimensional system is almost banded, with a nonzero top row representing the boundary condition, and has bandwidths $(1,1)$. Since $\varphi_n^j(y)$ is a polynomial of degree at most $n-1$, $\mR^j_n$ has at most $n$ nonzero entries. Therefore, unlike in a standard implementation of the ultraspherical spectral method, where adaptive QR is applied until the nontrivial part of the right-hand side has norm below a prescribed tolerance, \cref{systemleft} can be safely truncated to a finite system of size $n+2$ for $n \geq 2$.

Equation \cref{systemleft} is to be solved repeatedly for varying values of $n$ and for $j = 1, 2, \dots, r$. There are several techniques for accelerating the solution. First, a recombined basis that satisfies the Dirichlet boundary condition can be employed to transform \cref{systemleft} into a strictly banded system with bandwidths $(1,2)$, for which standard banded solvers, e.g., \textsc{LAPACK}'s subroutines \texttt{gbtrf} and \texttt{gbtrs}, can be used for significantly faster solutions. For details, see, for example, \cite{qin}. Second, if the algorithm is implemented in a matrix-oriented programming language, the $r$ $n\times n$ linear systems can be solved simultaneously by collecting the $r$ right-hand sides into a single $n\times r$ matrix and performing the QR factorization in a batched manner. Third, for a moderate $r$, the entire task is embarrassingly parallel.

Suppose that $f_j(y)$ is a polynomial in $y$ of degree at most $K$ for all $j$. Once $\varphi_n^j(y)$ is available for $j = 1, 2, \dots, r$, $\varphi_n(y)$ can be assembled following \cref{psin}, where the product of $f_j(y)$ and $\varphi_n^j(y)$ can be computed in $O(K n)$ flops via direct matrix-vector multiplication. Since we are working with variable $y$, the multiplication matrix here is exactly the one used in the standard ultraspherical spectral method. Finally, we approximate $\left(1+\psi(y)\right)^{\mu}$ by a Chebyshev expansion and compute the coefficients of $_{-1}\mathcal{I}_{x}^{\mu} \left[Q_n\right](\psi(y))$ following \cref{IQnl}. Note that these coefficients admit a dual interpretation---they are the Chebyshev coefficients in $y$ of $_{-1}\mathcal{I}_{x}^{\mu}\left[Q_n\right](\psi(y))$, and equivalently the TCP coefficients in $x$ of $_{-1}\mathcal{I}_{x}^{\mu}\left[Q_n\right](x)$.

Suppose that $\left(1+\psi(y)\right)^{\mu}$ admits a TCP approximation of degree at most $K$, as is usually the case. Since $\mR^j$ is upper triangular, the infinite-dimensional matrix $\mA$ is lower-banded with lower bandwidth $K+1$. In practice, what we construct is a truncated version of $\mA$. Thus, we let $N$ denote the truncation size of $\mA$ in the remainder of this paper. Solving the corresponding linear system by a direct method requires $\mO(K N^2)$ flops.

\cref{alg:prototype} summarizes the prototype algorithm without going into the details of the specific variable transform.

\begin{algorithm}[t!]
\caption{Prototype algorithm for constructing the spectral approximation to a left-sided FIO using TCPs.}\label{alg:prototype}
\begin{algorithmic}
\STATE{Construct a low-rank approximation of $G(y,t)$ in the form of \cref{Glr}.}
\STATE{Compute the TCP coefficients of $\left(1+\psi(y)\right)^{\mu}$.}
\FOR{$j = 1$ to $r$}
  \STATE{Compute $\varphi_n^j(y)$ for $n = 0, 1, 2$.}
  \FOR{$n = 3$ to $N$}
    \STATE{Compute the boundary value $\varphi_n^j(1)$.}
    \STATE{Solve \cref{rec} for the coefficients of $\varphi_n^j(y)$.}
  \ENDFOR
\ENDFOR
\FOR{$n = 0$ to $N$}
  \STATE{Compute the coefficients of $\varphi_n(y)$ following \cref{psin}.} 
  \STATE{Compute the coefficients of ${}_{-1}\mI_x^{\mu}[Q_n](\psi(y))$ following \cref{IQnl} and store them in the $n$th column of $\mA$.}
\ENDFOR
\end{algorithmic}
\end{algorithm}

\subsection{Right-sided FIOs}
Since the approximation to the right-sided FIO \cref{fior} is analogous to the left-sided one, we list only the key results that replace the corresponding steps in \cref{alg:prototype} to give the algorithm for the right-sided FIO. The right-sided counterparts of \cref{IQnl}, \cref{psi}, and \cref{G} are
\begin{align}
{}_x\mI_1^{\mu}[Q_n]\left(\psi(y)\right) = \frac{1}{\Gamma(1+\mu)}\left( \left(1-\psi(y)\right)^{\mu} + n\varphi_n(y) \right), \label{IQnr}
\end{align}
where $\varphi_n(y) = (y-1)\int_{0}^{1} t^{\mu} U_{n-1}\left(y+(1-y)t\right) G(y, t) \md t$ and 
\begin{align*}
G(y,t) = \left(\frac{\psi(y+(1-y)t)-\psi(y)}{t}\right)^{\mu}.
\end{align*}
The function $\varphi_n(y)$ can further be written as
\begin{align}
\varphi_n(y) = \left(y-1\right) \sum_{j=1}^{r} \sigma_j f_j(y) \varphi_{n}^{j}(y), \label{psinr}
\end{align}
where the moment $\varphi_n^j(y) = \int_{0}^{1} t^{\mu} U_{n-1}\left(y+(1-y)t\right)g_j(t)\md t$ satisfies the recurrence relation
\begin{align}
\left((1-y)\frac{\md}{\md y} + n\right)\varphi_{n+1}^j(y) = \left((1-y)\frac{\md}{\md y} - n\right)\varphi_{n-1}^j(y) + 2n\varphi_n^j(y).\label{recr}
\end{align}
To enforce a left Dirichlet boundary condition, we need to evaluate 
\begin{align}\label{psink1r}
\varphi_n^j(-1) = \int_{0}^{1}t^{\mu} U_{n-1}\left(2t-1\right)  g_j(t) \md t.
\end{align}

We solve \cref{recr} in an analogous manner using the ultraspherical spectral method to obtain the Chebyshev coefficients of $\{\varphi_n^j(y)\}_{n=0}^N$ for $j = 1, 2, \dots, r$. We then assemble $\varphi_n(y)$ following \cref{psinr}. Finally, we compute the Chebyshev coefficients of $\left(1-\psi(y)\right)^{\mu}$ and those of ${}_x\mathcal{I}_1^{\mu}[Q_n]\left(\psi(y)\right)$.

\subsection{Multiplication operator}
As we shall see in \cref{sec:example}, we need the matrix representation of multiplication by a TCP expansion for fractional calculus problems involving variable coefficients. Again, representing a given function by a TCP expansion is equivalent to approximating the transplanted function by a Chebyshev series in the transplanted variable, i.e.,
\begin{align*}
f(x) = \sum_{j=0}^{\infty} c_j Q_j(x) \Longleftrightarrow f(\psi(y)) = \sum_{j=0}^{\infty} c_j T_j(y).
\end{align*}
Thus, multiplication by a TCP series in $x$ amounts to multiplication by the corresponding Chebyshev series in $y$. The following lemma formalizes this in a rigorous manner.

\begin{lemma}\label{thm:M}
Multiplication of the infinite series $\sum_{m=0}^{\infty} c_m Q_m(x)$ by another infinite TCP series can be represented by the infinite-dimensional matrix
\begin{align*}
\mM = \frac{1}{2}\left[
\begin{pmatrix*}[r]
2c_0 & c_1 & c_2 & c_3 & \cdots\\
c_1 & 2c_0 & c_1 & c_2 & \ddots\\
c_2 & c_1 & 2c_0 & c_1 & \ddots\\
c_3 & c_2 & c_1 & 2c_0 & \ddots\\
\vdots & \ddots & \ddots & \ddots & \ddots
\end{pmatrix*} +
\begin{pmatrix*}
0 & 0 & 0 & 0 & \cdots\\
c_1 & c_2 & c_3 & c_4 & \ddots\\
c_2 & c_3 & c_4 & c_5 & \ddots\\
c_3 & c_4 & c_5 & c_6 & \ddots\\
\vdots & \ddots & \ddots & \ddots & \ddots
\end{pmatrix*}
\right],
\end{align*}
which coincides with the multiplication matrix for infinite Chebyshev series.
\end{lemma}

\begin{proof}
For any $k, l \in \mathbb{N}$,
\begin{align*}
Q_k(x) Q_l(x) &= T_k(y) T_l(y) = \frac{1}{2}\left(T_{k+l}(y) + T_{\lvert k-l\rvert}(y)\right) \\
&= \frac{1}{2}\left(Q_{k+l}(x) + Q_{\lvert k-l\rvert}(x)\right).
\end{align*}
Thus, the result follows from the standard multiplication formula for Chebyshev polynomials; see, e.g., \cite[\S 2.2]{olv}.
\end{proof}

\section{Variable transforms}\label{sec:vt}
So far, we have left the variable transform unspecified. In this section, we instantiate the proposed framework using algebraic and exponential transforms.

\subsection{Algebraic transforms}
\begin{figure}[t!]
\centering
\subfloat[]{\label{fig:fa}
  \includegraphics[width=0.323\linewidth]{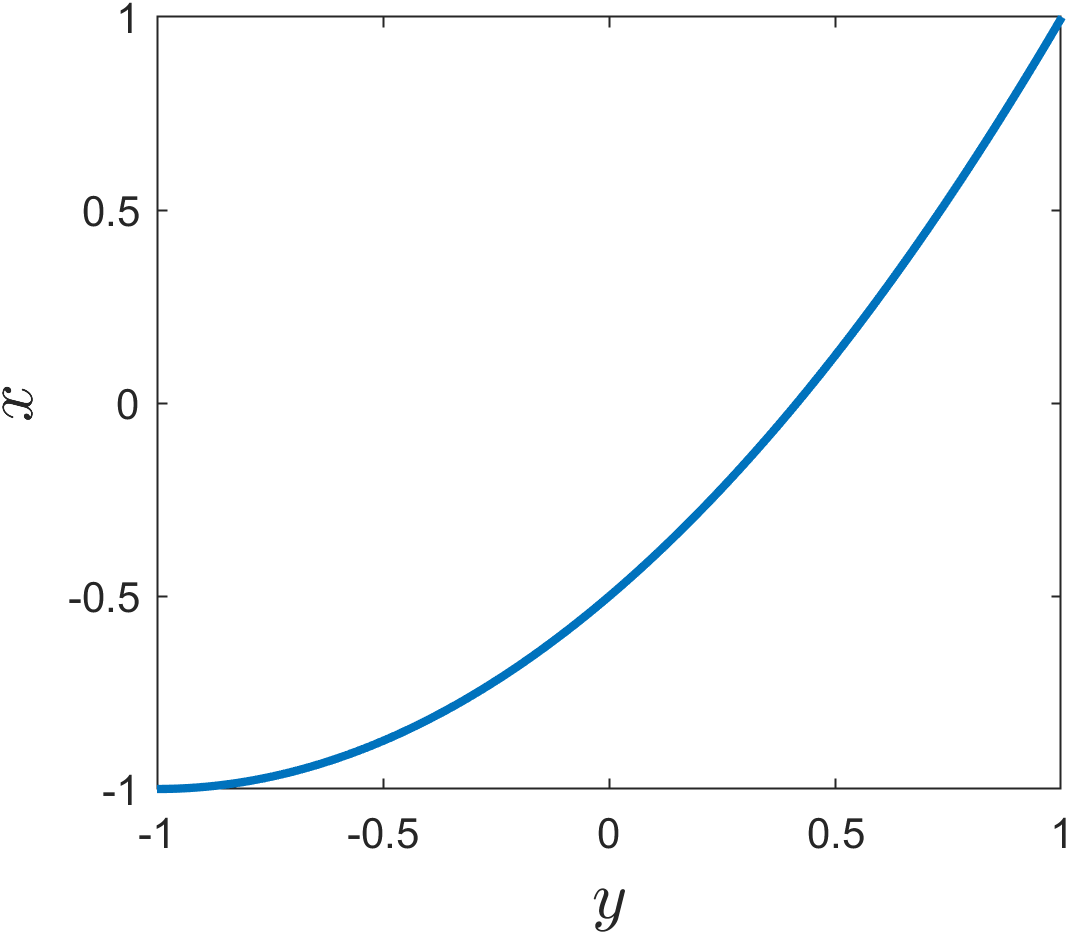}
}
\hfill
\subfloat[]{\label{fig:jfp}
  \includegraphics[width=0.6\linewidth]{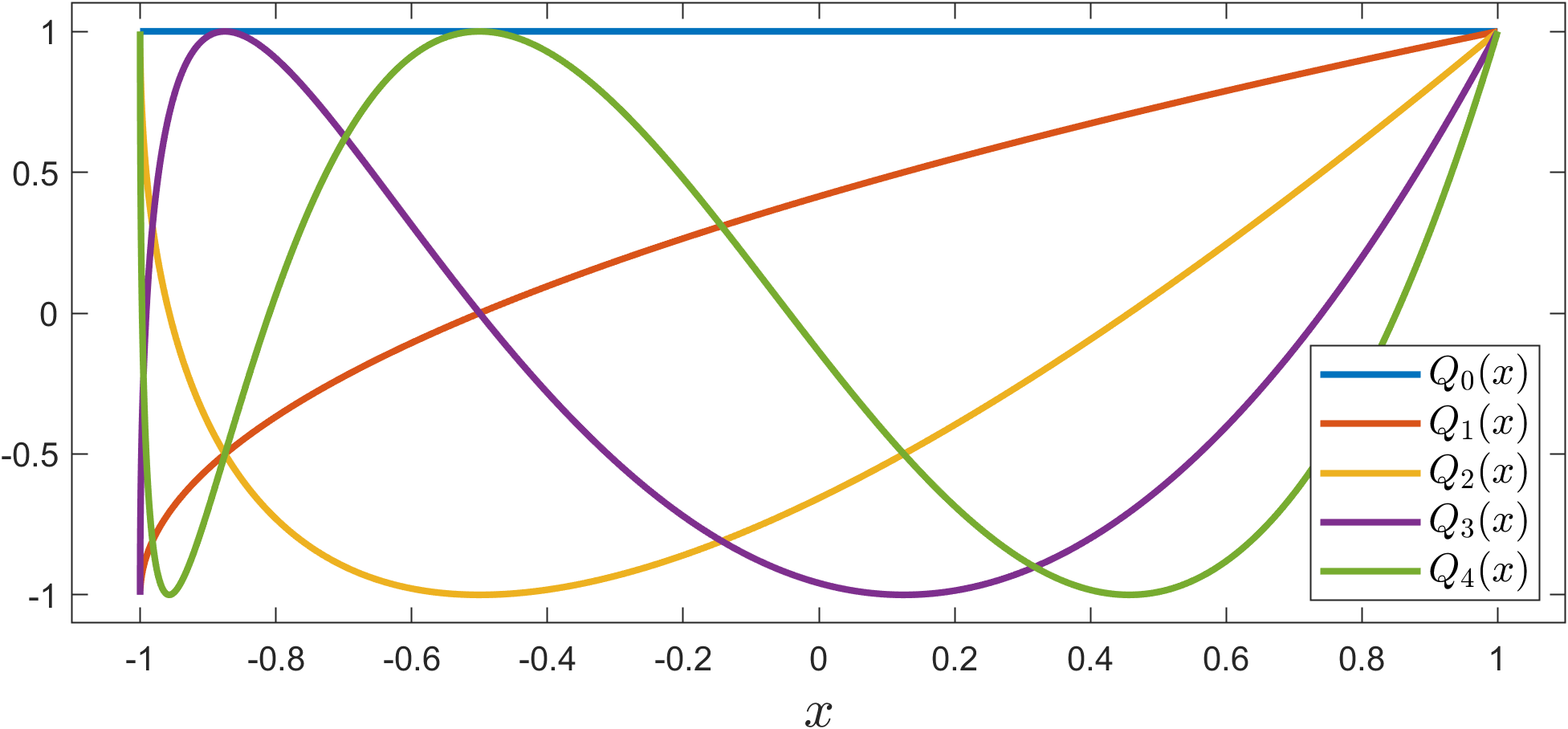}
}
\caption{The algebraic transform \cref{fa} for $\alpha = 0$ and $\beta = 1/2$: (a) the forward transform \cref{fa} and (b) the first five JFPs.}\label{fig:algebraic}
\end{figure}

By composing classical orthogonal polynomials on $[-1,1]$ with a fractional power of a polynomial that maps $[-1,1]$ onto itself, we obtain the so-called Jacobi fractional (or M\"{u}ntz) polynomials. For example, in \cite{liu} the algebraic transform \cref{algebraic} is composed with the Chebyshev polynomials to obtain the Chebyshev version of the JFP \cref{jfp} for the left-sided FIO \cref{fiol}. For \cref{algebraic}, the forward mapping is
\begin{align}
x = \psi(y) = 2\left(\frac{1+y}{2}\right)^{\frac{1}{\beta}} - 1, \label{fa}
\end{align}
which is shown in \cref{fig:fa} for $\beta = 1/2$. The first five JFPs \cref{jfp} for $\alpha = 0$ and $\beta = 1/2$ are shown in \cref{fig:jfp}. Substituting \cref{fa} into \cref{G} gives
\begin{align*}
G(y,t) = 2^{\mu(1-\frac{1}{\beta})}(1+y)^{\frac{\mu}{\beta}} \left(\frac{1-(1-t)^{\frac{1}{\beta}}}{t}\right)^{\mu} = \sigma_1 f_1(y) g_1(t),
\end{align*}
where $\sigma_1 = 1$ and
\begin{align*}
f_1(y) = 2^{\mu(1-\frac{1}{\beta})}(1+y)^{\frac{\mu}{\beta}}, \qquad
g_1(t) = \left(\frac{1-(1-t)^{\frac{1}{\beta}}}{t}\right)^{\mu}.
\end{align*}
Thus, it follows from \cref{psin} that $\varphi_n(y) = (1+y) f_1(y)\varphi_n^1(y)$ and
\begin{align*}
\varphi_n^1(y)
= \int_{0}^{1} t^{\mu} U_{n-1}\left(y - (1+y)t\right) \left(\frac{1-(1-t)^{\frac{1}{\beta}}}{t}\right)^{\mu} \md t,
\end{align*}
which satisfies the three-term recurrence relation \cref{rec}.

By the chain rule, we have
\begin{align}
\frac{\md}{\md y} \varphi_n^1(y)
= \beta\left(\frac{1+y}{2}\right)^{1-\frac{1}{\beta}} \frac{\md}{\md x} \bigl(\varphi_n(y)\bigr) - \frac{\varphi_n^1(y)}{1+y}. \label{chain}
\end{align}
Substituting \cref{chain} into \cref{rec}, we arrive at
\begin{align*} 
& (n+1)\left(\frac{1+x}{n+1}\frac{\md}{\md x} - \beta\right)\Bigl((1+y)\varphi_{n+1}^1(y)\Bigr) \\ &= (n-1)\left(\beta+\frac{1+x}{n-1}\frac{\md}{\md x}\right)\Bigl((1+y)\varphi_{n-1}^1(y)\Bigr) + 2\beta n(1+y)\varphi_n^1(y), 
\end{align*}
which, after algebraic manipulation, reduces to the three-term recurrence relation given in \cite[Theorem 2.1]{liu}. The remainder of the construction follows \cite[\S 2]{liu} or \cref{sec:left}.

The spectral approximations based on the algebraic transform are typically simple and are easy to implement. More importantly, they lead to stable and fast spectral methods for fractional calculus whenever they are applicable. Two disadvantages of the algebraic transform are readily observed: (1) it can only handle functions with a singularity at one endpoint, not both; (2) it does not allow spectral approximations to sums of multiple FIOs whose fractional orders do not satisfy the compatibility conditions for \cref{fie}; we refer to such fractional orders as incompatible. When the JFP spectral method is not applicable, the algebraic transform can be replaced by exponential transforms.

\subsection{Exponential transforms}
The limitations of algebraic transforms, and consequently those of the JFP spectral methods, can be overcome by a double-sided exponential transform. We consider, as an example, a rescaled double-exponential transform, i.e.,
\begin{subequations}\label{de}
\begin{align}
x = \psi(y) = \tanh\!\left(\frac{\pi}{2}\sinh(\omega y)\right), \qquad y \in (-\infty, \infty), \label{fde}
\end{align}
and its inverse
\begin{align}
y = \psi^{-1}(x) = \frac{1}{\omega} \arcsinh\!\left(\frac{2}{\pi} \arctanh(x)\right), \qquad x \in (-1,1). \label{ide}
\end{align}
\end{subequations}
These transforms differ from the standard double-exponential transforms \cite{tak} by the presence of the scaling factor $\omega$. Apparently, $\psi(y)$ does not map $[-1, 1]$ onto itself. We take advantage of $\omega$ to make it numerically so.

\begin{figure}[t!]
\centering
\subfloat[]{
  \includegraphics[width=0.323\linewidth]{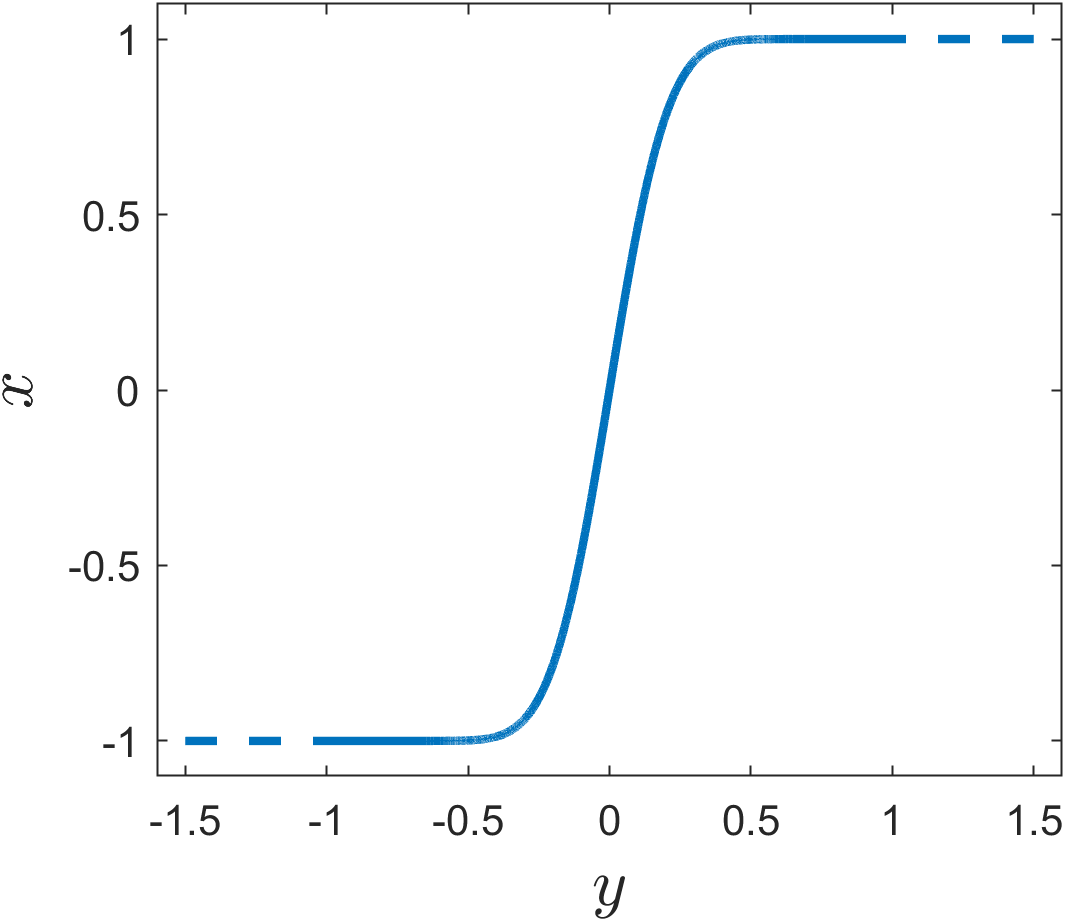}
}
\hfill
\subfloat[]{
  \includegraphics[width=0.6\linewidth]{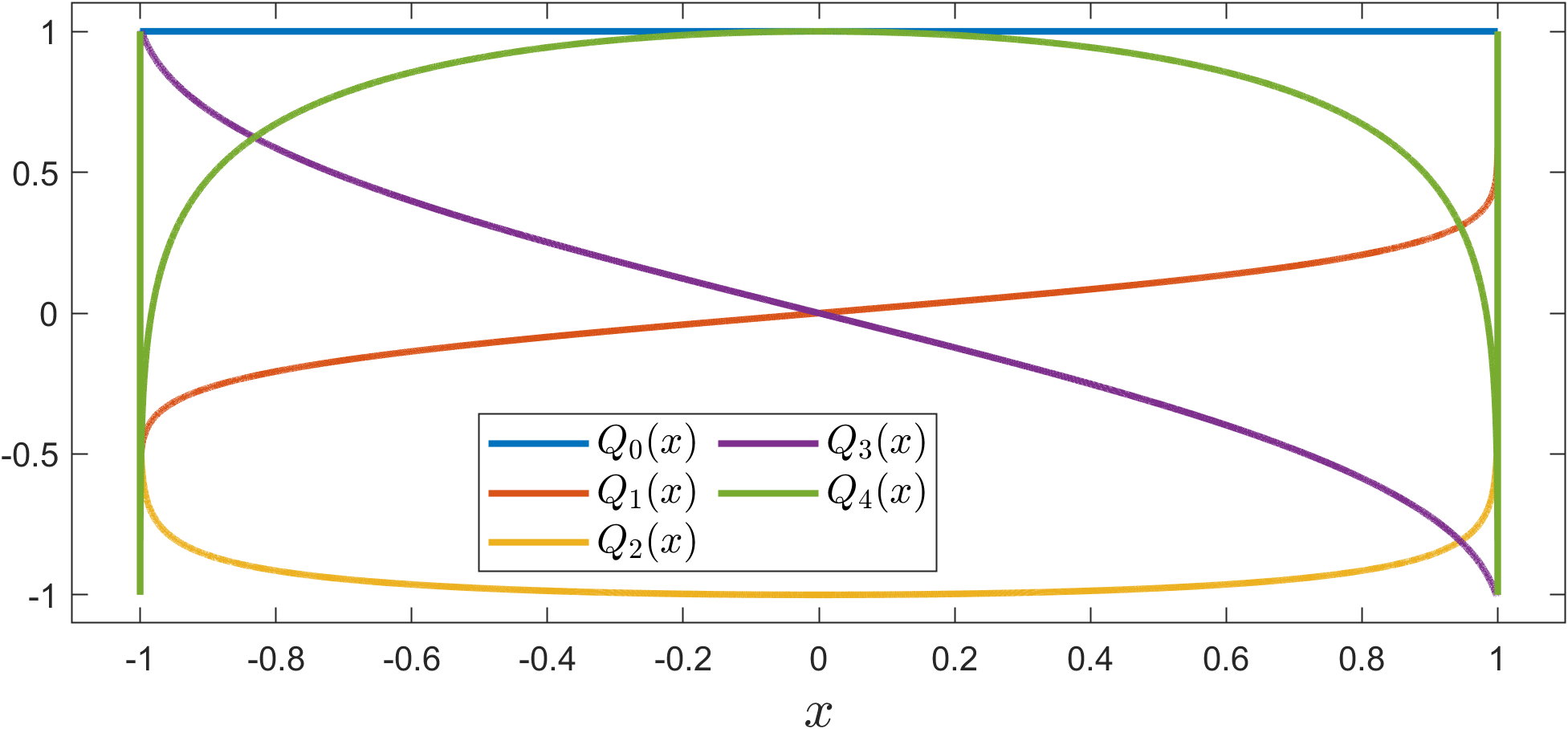}
}
\caption{The double-exponential transform: (a) the forward transform \cref{fde} with $\omega = 3.154$, chosen so that \cref{beta1} holds; (b) the first five DETCPs. These plots are intended for comparison with \cref{fig:algebraic}.}
\end{figure}

\subsubsection{Determining $\omega$}
Noting that $\lim_{\omega \to \infty} \psi(\pm 1) = \pm 1$, we choose a finite $\omega$ such that $\psi(\pm 1)$ approximates $\pm 1$ with an error no greater than machine epsilon, i.e.,
\begin{align}
\left| \psi(\pm 1) \mp 1 \right| \le \ema. \label{beta1}
\end{align}
In double-precision floating-point arithmetic, $\omega = 3.154$ is roughly the smallest value that satisfies this condition.

Suppose that $f(x)$ is a function defined on $[-1,1]$ that exhibits a weak singularity of order $\gamma$ at $x = -1$, i.e., $f(x) \approx (1+x)^{\gamma}$ as $x \to -1$. In addition to \cref{beta1}, we require that when the transplanted function $f(\psi(y))$ is evaluated at and near $y=-1$, the computed value deviates from the true value by no more than $\mathcal{O}(\ema \|f\|_{\infty})$. That is,
\begin{align}
f\bigl(\psi(-1)\bigr) - f(-1) \approx \bigl(1+\psi(-1)\bigr)^{\gamma} \le \ema \|f\|_{\infty}. \label{beta2}
\end{align}
Together with the requirement $\omega \ge 3.154$, we obtain
\begin{subequations}\label{beta3}
\begin{align}
\omega = \max\left(3.154,\,
\operatorname{arcsinh}\!\left(\frac{1}{\pi}\log\!\left(\frac{2}{(\ema\|f\|_{\infty})^{1/\gamma}} - 1\right)\right)\right). 
\end{align}
In practice, the logarithmic term is evaluated as
\begin{align}
\log\!\left(\frac{2}{(\ema\|f\|_{\infty})^{1/\gamma}} - 1\right)
= \log 2 - \frac{\log(\ema\|f\|_{\infty})}{\gamma}
+ \log\!\left(1 - \frac{(\ema\|f\|_{\infty})^{1/\gamma}}{2}\right)
\end{align}
\end{subequations}
to avoid floating-point overflow. By symmetry, an analogous analysis at the other endpoint leads to the same requirement.

With such a choice of $\omega$, we assume henceforth that $\psi(\pm 1) = \pm 1$, and refer to the corresponding TCPs as double-exponential transplanted Chebyshev polynomials (DETCPs).

In addition, we evaluate the inverse transform as
\begin{align*}
\psi^{-1}(x) = \min\!\left(1,\max\!\left(-1,\frac{1}{\omega}\arcsinh\!\left(\frac{2}{\pi}\arctanh(x)\right)\right)\right),
\end{align*}
to ensure that the returned value lies in $[-1,1]$.

\cref{fig:beta} shows the error in the DETCP approximations to $f(x) = (1+\psi(y))^{\mu}$ as $\omega$ increases for various $\mu$. The dashed vertical lines, which indicate the values of $\omega$ determined by \cref{beta3}, pinpoint the threshold beyond which the errors reach the level of machine precision.

\begin{figure}[t!]
\centering
\subfloat[]{\includegraphics[width=0.47\linewidth]{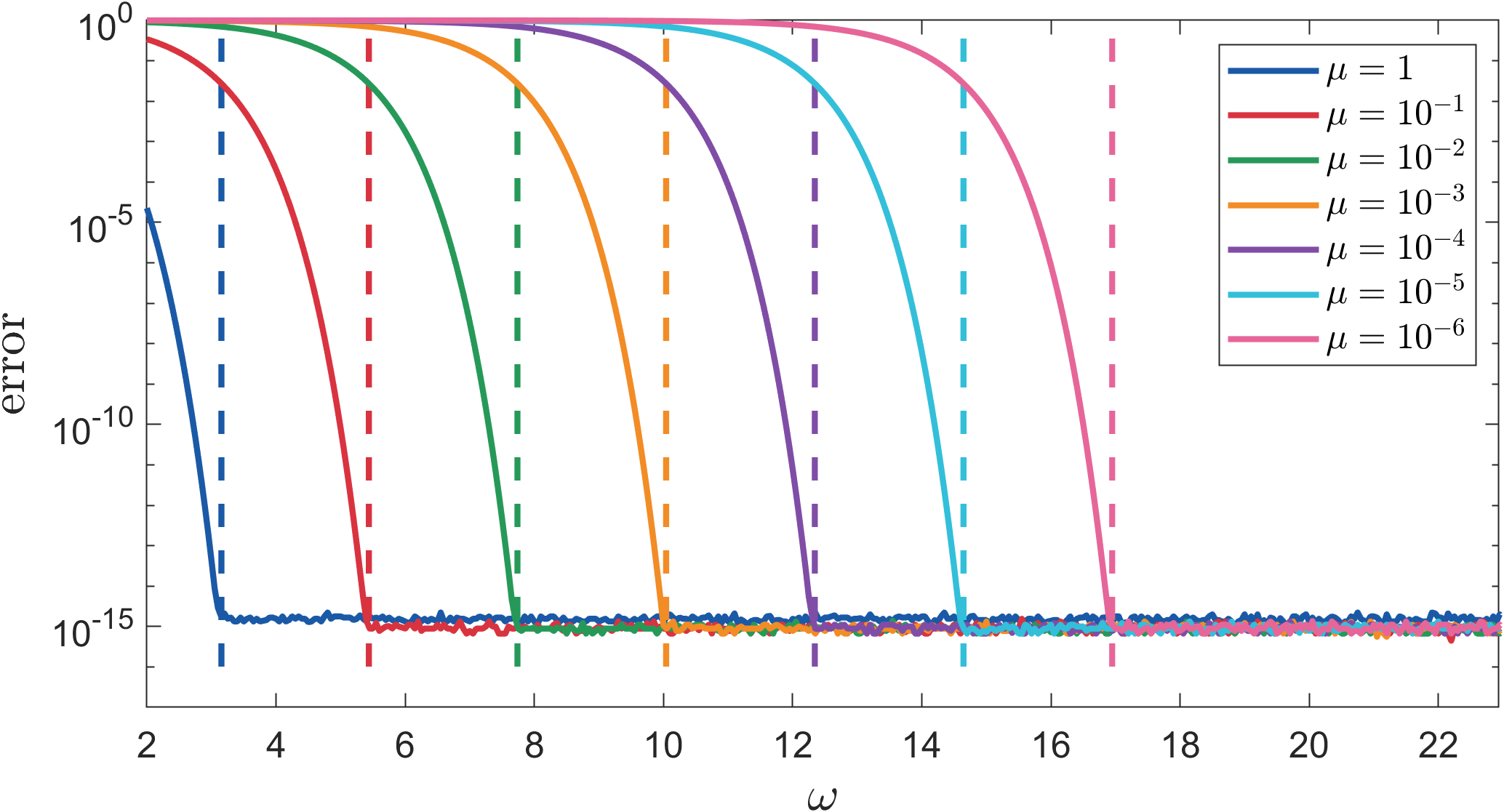}\label{fig:beta}}
\hfill
\subfloat[]{\includegraphics[width=0.47\linewidth]{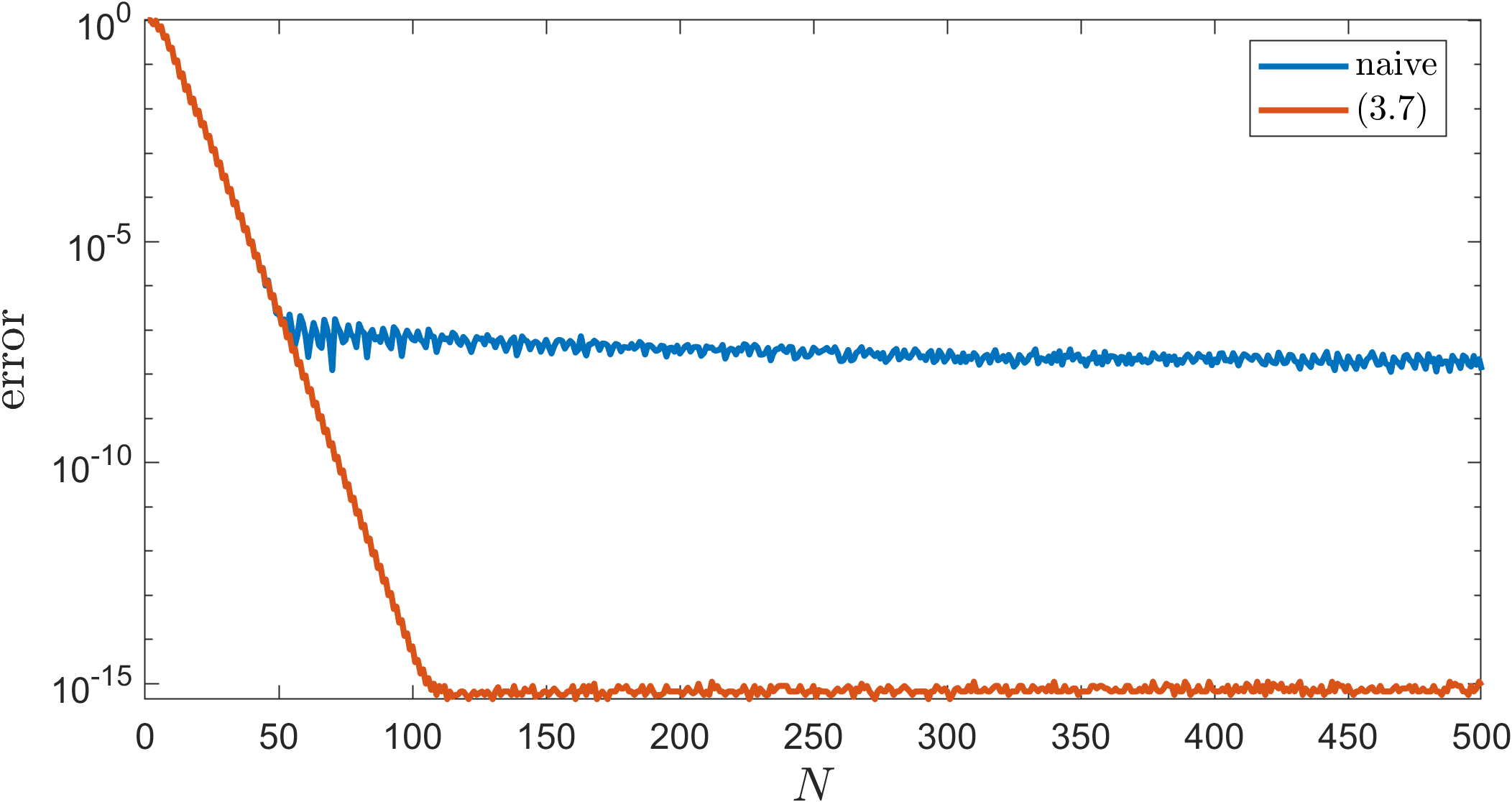}\label{fig:approx1}}\\
\subfloat[]{\includegraphics[width=0.47\linewidth]{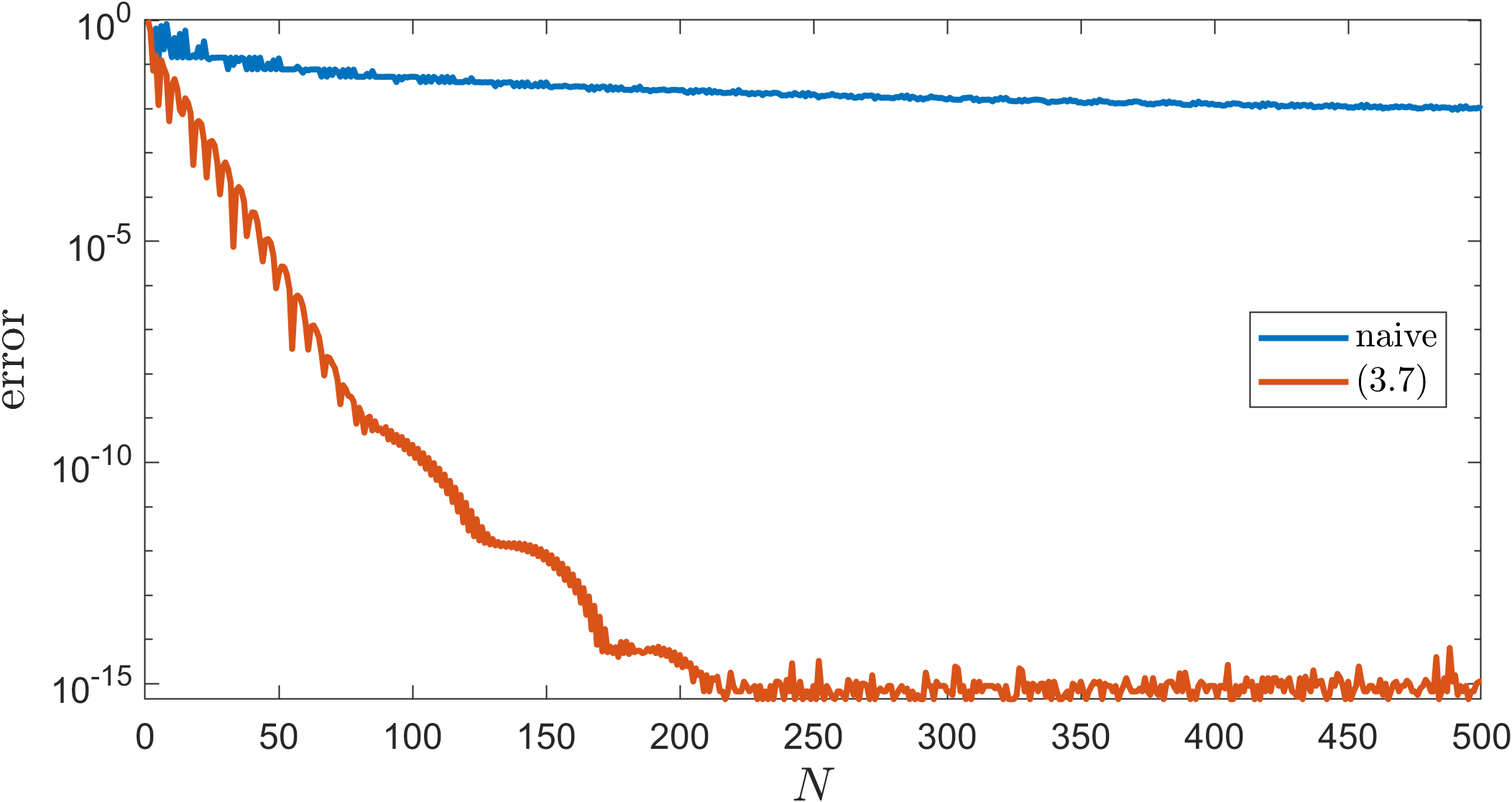}\label{fig:approx2}}
\hfill
\subfloat[]{\includegraphics[width=0.47\linewidth]{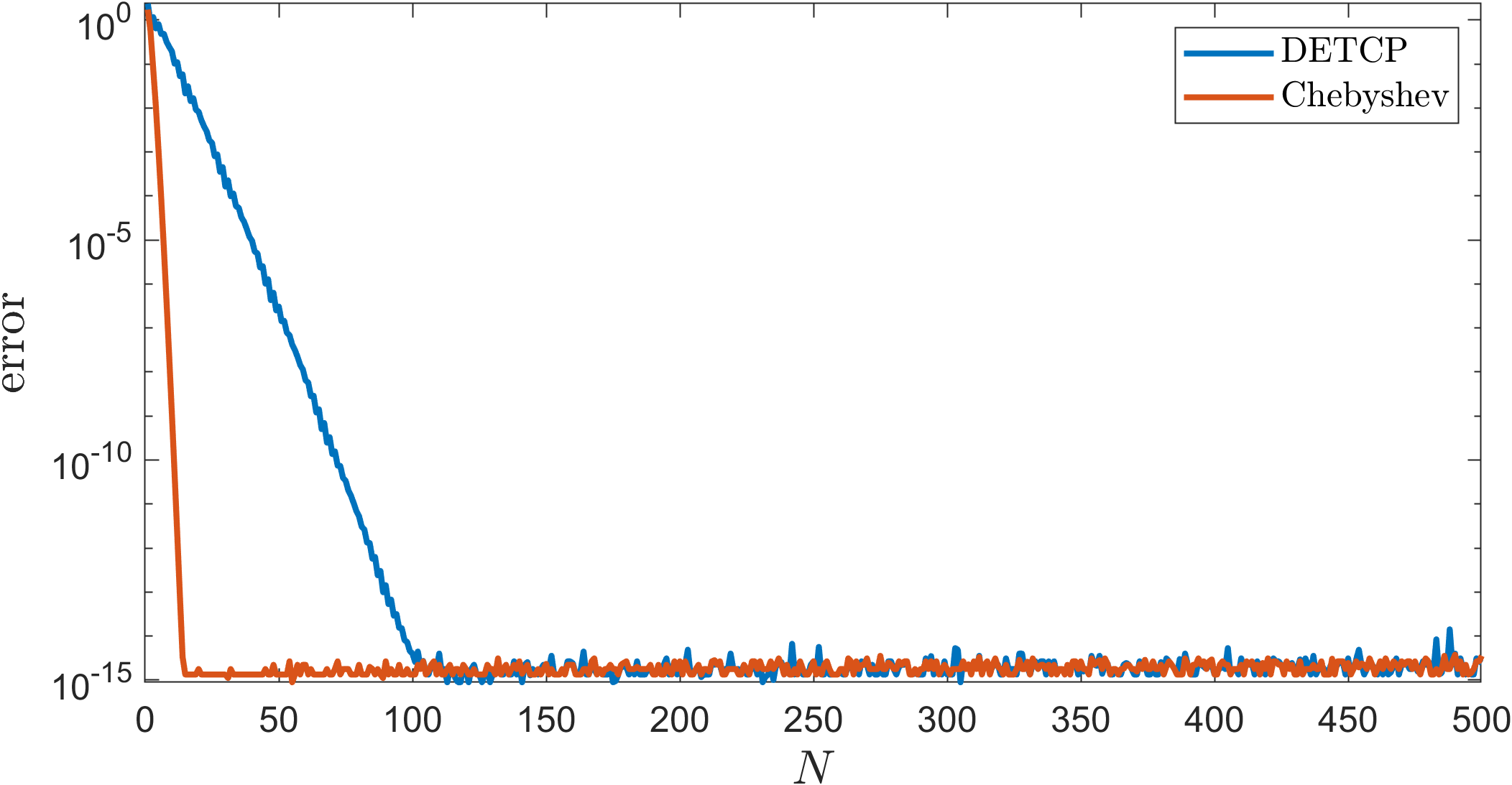}\label{fig:approx3}}
\caption{(a) The parameter $\omega$ versus the errors in the computed values of $(1+\psi(y))^{\mu}$ for various fractional order $\mu$. The vertical line marks the values determined by \cref{beta3}. Error in the DETCP expansion of (b) $f(x) = (1-x)^{\frac{1}{3}}(1+x)^{\frac{1}{2}}$, (c) $(1+x)^{10^{-5}}$, and (d) $f(x) = e^x$. For these functions, $\omega = 4.238, 14.646, 3.154$, respectively, and the errors are estimated by evaluating the pointwise error at $2\times 10^4$ evenly spaced points in $[-1, 1]$.} \label{fig:approx}
\end{figure}

\subsubsection{Implementation details}
We now discuss implementation details for constructing the double-exponential-based spectral approximation (DESA).

\paragraph{Evaluation of key quantities}
We have to evaluate $(1 \pm \psi(y))^{\kappa}$ when constructing the DETCP expansion of the $(1 \pm \psi(y))^{\mu}$ term in \cref{IQnl} and those of the variable coefficients or right-hand side functions that involve fractional-power of $1 \pm \psi(y)$. A naive evaluation of $(1 \pm \psi(y))^{\kappa}$ may lead to severe cancellation errors. To address this issue, we evaluate the logarithm of $1 \pm \psi(y)$ as
\begin{align}
\log\left(1 \pm \psi(y)\right)
= \log 2 - \log\left(1 + \exp\left(\mp \pi \sinh(\omega y)\right)\right), \label{1pmphi}
\end{align}
then multiply by $\kappa$ and exponentiate.

\paragraph{DETCP expansion of a given function}
We often need to construct the DETCP expansion of a given function, for example, the variable coefficients $a_k(x)$ and $b_k(x)$ and the right-hand side function in \cref{fie}. Since expressing a given function $f(x)$ in terms of the basis $\{Q_n(x)\}_{n=0}^{N}$ can be achieved by interpolating $f(\psi(y))$ at Chebyshev points on $[-1,1]$, the DETCP expansion can be computed at essentially the same cost as the standard Chebyshev interpolation, namely $\mathcal{O}(N \log N)$ using the fast cosine transform.

\cref{fig:approx1,fig:approx2,fig:approx3} demonstrate the approximation of three functions: $f(x) = (1-x)^{\frac{1}{3}}(1+x)^{\frac{1}{2}}$, $f(x) = (1+x)^{10^{-5}}$, and $f(x) = e^{x}$ on $[-1,1]$. These examples represent, respectively, a typical fractional-power function with singularities at both endpoints, a function with an extremely weak singularity, and a smooth (indeed, analytic) function. 

On the one hand, if $1 \pm \psi(y)$ is evaluated naively, the exponential convergence of the approximation error ceases prematurely. With the stabilization technique in \cref{1pmphi}, the error in the DETCP approximation decays to machine precision. On the other hand, since $e^{x}$ is smooth, its DETCP approximation generally requires more basis functions than the standard Chebyshev approximation, as expected.

\paragraph{Bivariate Chebyshev expansion of $G(y, t)$}
For exponential transforms, particularly the double-exponential transform, $G(y, t)$ is usually too complicated to allow separation of variables in $y$ and $t$, in contrast to the algebraic transform, for which such a separation is possible. Hence, we resort to numerical approximation to obtain \cref{Glr}. However, direct evaluation of \cref{G} by simply substituting \cref{fde} usually results in severe cancellation errors, particularly for small $t$. To address this, we rewrite
\begin{align}\label{Gl}
G(y, t)& = \exp\Bigg( \mu \bigg( \log\left(\frac{\pi\omega(1+y)}{2}\right) + \log\left(\cosh\left(\frac{\omega(y+\xi_1)}{2}\right)\right) + \log\left(\operatorname{sinhc}(\xi_3)\right) \nonumber \\
 &+ \log\left(\operatorname{sinhc}(\xi_2)\right) - \log\left(\cosh\left(\frac{\pi}{2}\sinh(\omega y)\right)\right) - \log\left(\cosh\left(\frac{\pi}{2}\sinh(\omega \xi_1)\right)\right) \bigg) \Bigg),
\end{align}
where
\begin{align*}
\xi_1  = y-(1+y)t, \qquad \xi_2 = \frac{\omega(1+y)t}{2}, \qquad \xi_3 = \pi \cosh\left(\frac{\omega(y+\xi_1)}{2}\right)\sinh(\xi_2),
\end{align*}
and $\operatorname{sinhc}(y) = \sinh(y)/y$ is the hyperbolic sinc function, with value $1$ at $y=0$. 

\begin{figure}
\centering
\subfloat[]{\label{fig:G}
  \includegraphics[width=0.44\linewidth]{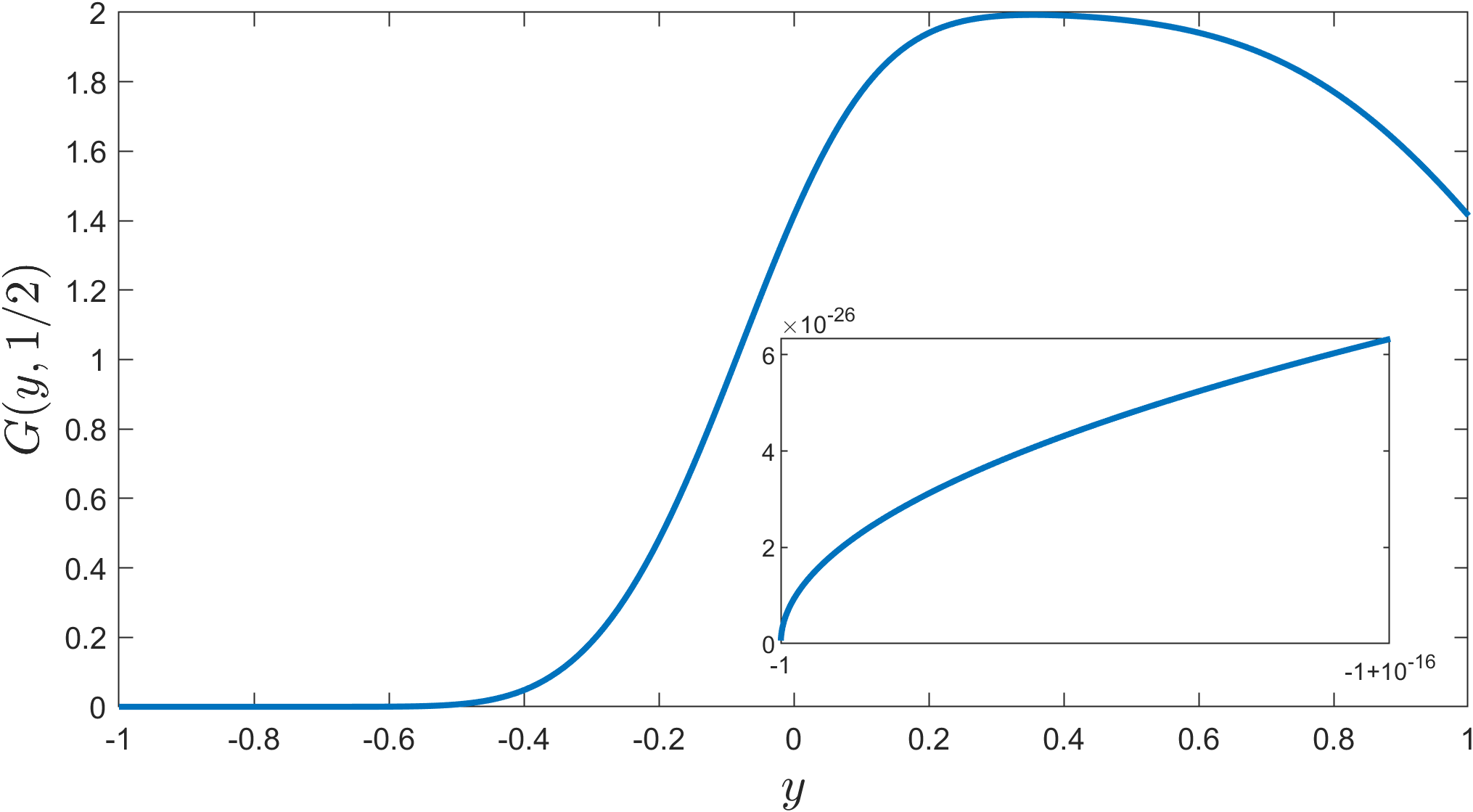}
}
\hfill
\subfloat[]{\label{fig:G0}
  \includegraphics[width=0.465\linewidth]{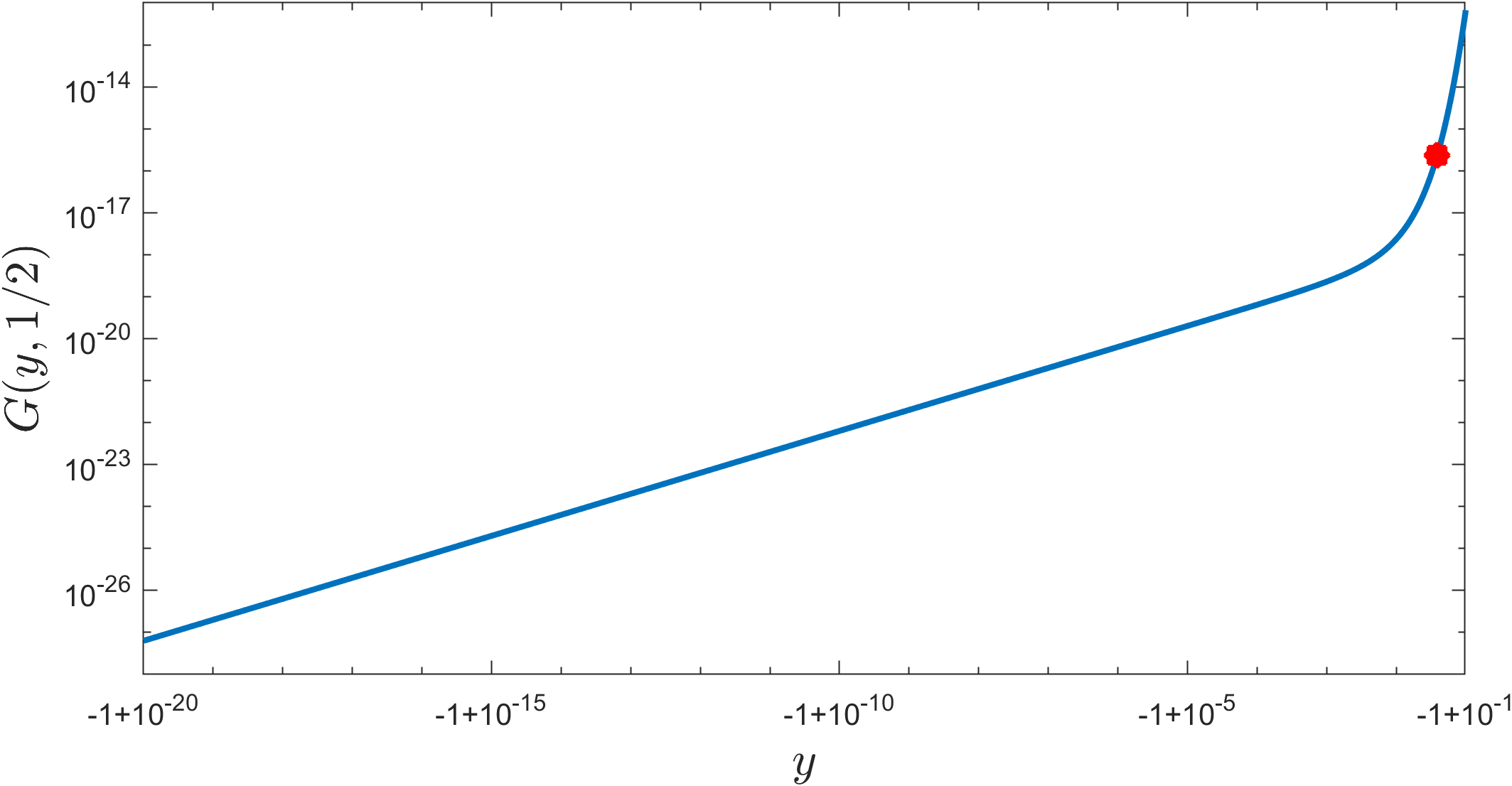}
}
\caption{(a) $G(y, 1/2)$ is smooth except at $y = -1$, where it has a weak singularity of order $(1+y)^{\mu}$. (b) A log--log plot of $G(y, t)$ near $y = -1$.}
\end{figure}

In fact, $G(y, t)$ is smooth except for a weak singularity at $y = -1$, because $\partial G(y,t) / \partial y \to \infty$ as $y \to -1$. \cref{fig:G} shows a plot of $G(y,1/2)$ for $\mu = 1/2$ and $\omega = 4$, together with a close-up near $y = -1$. This singularity does not, however, prevent us from obtaining a global bivariate Chebyshev approximation on $(y, t) \in [-1, 1]\times [0, 1]$. To see why, we expand $G(y,t)$ about $y=-1$ and obtain
\begin{align}
G(y, t)=\left(\psi^{\prime}(-1)\right)^\mu(1+y)^\mu(1+\mO(1+y)), \label{Gasymptotic}
\end{align}
as $y \to -1$. Since $\psi^{\prime}(-1) \sim 2 \pi \omega \cosh (\omega) \exp (-\pi \sinh \omega)$, it follows from the fact that $1+\psi(-1) \sim 2 \exp (-\pi \sinh \omega)$ and \cref{beta2} that
\begin{align*}
\left(\psi^{\prime}(-1)\right)^\mu \lesssim(\pi \omega \cosh \omega)^\mu \|f\|_{\infty} \ema,
\end{align*}
or, simply $\left(\psi^{\prime}(-1)\right)^\mu = \mO(\ema)$. Therefore, for $y$ in a sufficiently large neighborhood of $-1$, \cref{Gasymptotic} implies that $G(y,t)$ is of magnitude $\mO(\ema)$ or smaller, and is thus negligible compared with its values elsewhere in the domain. Consequently, the singularity is effectively invisible to a global approximation procedure. This can be seen from \cref{fig:G0}, where $G(y,1/2)$ is plotted on a log--log scale. Up to approximately $y=-0.961$ (the horizontal coordinate of the red dot), $G(y,1/2)$ remains negligible compared with its overall magnitude, which is about $2.0$. The vertical coordinate of the red dot marks the value $2\ema$, below which $G(y,1/2)$ is effectively treated as zero in numerical approximation.

Since $G(y,t)$ is numerically smooth, it can be well approximated by a bivariate Chebyshev expansion of the form \cref{Glr}, where $f_j(y)$ and $g_j(t)$ are Chebyshev series
\begin{align*}
f_j(y) = \sum_{k=0}^K a_{kj} T_k(y), ~~~~g_j(t) = \sum_{l=0}^L b_{lj} T_l(2t-1).
\end{align*}
Such a low-rank bivariate Chebyshev expansion can be constructed using adaptive cross approximation (ACA) \cite{tow} at a cost of $\mO(r^3+r^2(K+L))$ flops. Our extensive experiments show that, for most problems, it suffices to take $K=L=100$ and $r = 30$, except in the case of extremely small $\mu$; see \cref{sec:abel}.

For the right-sided FIO, the corresponding cancellation-free expression is
\begin{align}
G(y,t) &= \exp\Bigg( \mu \bigg( \log\left(\frac{\pi\omega(1-y)}{2}\right) + \log\left(\cosh\left(\frac{\omega(y+\xi_1)}{2}\right)\right) + \log\left(\operatorname{sinhc}(\xi_3)\right) \nonumber\\
&+ \log\left(\operatorname{sinhc}(\xi_2)\right) - \log\left(\cosh\left(\frac{\pi}{2}\sinh(\omega y)\right)\right) - \log\left(\cosh\left(\frac{\pi}{2}\sinh(\omega \xi_1)\right)\right) \bigg) \Bigg), \label{Gr}
\end{align}
where
\begin{align}
\xi_1 = y+(1-y)t, \quad \xi_2 = \frac{\omega(1-y)t}{2}, \quad \xi_3 = \pi \cosh\left(\frac{\omega(y+\xi_1)}{2}\right)\sinh(\xi_2).
\end{align}

\paragraph{Initial values and boundary condition for $\varphi_n^j(y)$}
Let 
\begin{align}
h_l = \int_{0}^{1} t^{\mu} T_l(2t-1) \md t, \label{moments}
\end{align}
which can be computed in $\mathcal{O}(L)$ flops using Gauss--Jacobi quadrature. The following lemma expresses the first three $\varphi_n^j(y)$ in terms of $\{b_{lj}\}_{l = 0}^{L}$ and $\{h_l\}_{l = 0}^{L+1}$.

\begin{lemma}\label{lem:psink012}
For all $j$,
\begin{align*}
\varphi_0^j(y) &= 0, \qquad 
\varphi_1^j(y) = \int_{0}^{1} t^{\mu} g_j(t) \md t = \sum_{l=0}^{L} h_l b_{lj}, \\
\varphi_2^j(y) &= 2\int_{0}^{1} t^{\mu} \left(y - (y+1)t\right) g_j(t) \md t = 2\left(\sum_{l=0}^{L+1} h_l (b_{lj} - \tilde{b}_{lj})\right) T_1(y) - 2\sum_{l=0}^{L+1} h_l \tilde{b}_{lj},
\end{align*}
where 
\begin{align*}
\tilde{b}_{0j} = \frac{1}{4}\left(2b_{0j} + b_{1j}\right), ~~
\tilde{b}_{1j} = \frac{1}{4}\left(2b_{0j} + 2b_{1j} + b_{2j}\right),~~
\tilde{b}_{lj} = \frac{1}{4}\left(b_{l-1,j} + 2b_{lj} + b_{l+1,j}\right)
\end{align*}
for $2\le l\le L+1$. Here, $b_{lj}=0$ for $l > L$.
\end{lemma}

\begin{proof}
We omit the proofs for $\varphi_0^j(y)$ and $\varphi_1^j(y)$. By definition, 
\begin{align}
\varphi_2^j(y) = 2y\int_{0}^{1} t^{\mu} g_j(t) \md t - 2(y+1) \int_{0}^{1} t^{\mu+1}g_j(t) \md t. \label{psij2}
\end{align}
Since $t = ( T_1(2t-1)+1 )/2$, 
\begin{align*}
t\, g_j(t)
&= \frac{1}{2}\left(\sum_{l=0}^{L} b_{lj} T_l(2t-1) + \frac{1}{2}\sum_{l=0}^{L}b_{lj} \left(T_{l+1}(2t-1)+T_{|l-1|}(2t-1)\right)\right) \\
&= \sum_{l=0}^{L+1}\tilde{b}_{lj} T_l(2t-1),
\end{align*}
where $\tilde{b}_{lj}$ are given above. The integrals in \cref{psij2} can then be evaluated as
\begin{align*}
\int_{0}^{1} t^{\mu} g_j(t) \md t = \sum_{l=0}^{L} h_l b_{lj}, \qquad
\int_{0}^{1} t^{\mu+1} g_j(t) \md t = \sum_{l=0}^{L+1} h_l \tilde{b}_{lj},
\end{align*}
which, upon substitution into \cref{psij2}, yields the expression for $\varphi_2^j(y)$.
\end{proof}

For right-sided FIOs, the first three terms of $\{\varphi_n^j(y)\}_{n = 0}^N$ are
\begin{align}
\varphi_0^j(y) = 0, ~ \varphi_{1}^{j}(y) = \sum_{l=0}^{L} h_l b_{lj}, ~ \varphi_{2}^{j}(y) = 2\left(\sum_{l=0}^{L+1} h_l (b_{lj} - \tilde{b}_{lj})\right) T_1(y) + 2\sum_{l=0}^{L+1} h_l \tilde{b}_{lj}, \label{psink012r}
\end{align}
where $h_l$ and $\tilde{b}_{lj}$ are defined as in \cref{lem:psink012}.

For the boundary condition, the integrals appearing on the right-hand sides of \cref{psink1,psink1r} can be computed in $\mathcal{O}(n)$ flops using Gauss--Jacobi quadrature.

\paragraph{Detailed description of the construction algorithm}
\begin{algorithm}[t!]
\caption{Algorithm for constructing the spectral approximation to a left- or right-sided FIO with DETCPs.}\label{alg:de}
\begin{algorithmic}
\STATE{Determine $\omega$ by \cref{beta3}.\hfill $\rhd \ \mO(1)$}
\STATE{Compute the bivariate Chebyshev expansion of $G(y,t)$ using \cref{Gl} or \cref{Gr}.}
\STATE{\hfill$\rhd\ \mO(r^3+r^2(K+L))$}
\STATE{Compute the DETCP coefficients of $\left(1\pm \psi(y)\right)^{\mu}$ via \cref{1pmphi}.\hfill $\rhd\ \mO(N\log N)$}
\STATE{Compute the moments $h_l$ defined in \cref{moments} for $l = 0, 1, \dots, L+1$.\hfill $\rhd\ \mO(L^2)$}
\FOR{$j = 1$ to $r$}
  \STATE{Compute $\varphi_n^j$ for $n = 0, 1, 2$ using \cref{lem:psink012} or \cref{psink012r}.\hfill $\rhd\ \mO(L)$}
  \FOR{$n = 3$ to $N$}
    \STATE {Compute $\varphi_n^j(1)$ via \cref{psink1} or $\varphi_n^j(-1)$ via \cref{psink1r}. \hfill $\rhd\ \mO(n)$}
    \STATE {Solve \cref{rec} or \cref{recr} for the coefficients of $\varphi_n^j(y)$. \hfill $\rhd\ \mO(n)$\hspace*{0.3\algorithmicindent}}
  \ENDFOR
\ENDFOR
\FOR{$n = 0$ to $N$}
  \STATE{Compute the coefficients of $\varphi_n(y)$ following \cref{psin} or \cref{psinr}.\hfill $\rhd\ \mO(rKN)$}
  \STATE{Compute the coefficients of ${}_{-1}\mI_x^{\mu}[Q_n](\psi(y))$ following \cref{IQnl} or ${}_{x}\mI_{1}^{\mu}[Q_n](\psi(y))$ following \cref{IQnr}, and store them in the $n$th column of $\mA$.\hfill $\rhd\ \mO(N)$\hspace*{0.3\algorithmicindent}}
\ENDFOR
\vspace{0.5em}
\hrule
\vspace{0.5em}
\STATE{Total complexity: \hfill $\mO\Bigl(rKN^2 + L^2 + r^2(K+L) + r^3 \Bigr)$}
\end{algorithmic}
\end{algorithm}

The complete algorithm for constructing DESAs is given in \cref{alg:de}, together with the stepwise and overall computational complexities. The cost of approximating $G(y, t)$ can be regarded as a precomputation step that is performed only once and depends solely on the exponential transform and the value of $\mu$. For a moderate $\omega$, this cost constitutes a relatively negligible overhead.

\subsection{Other variable transforms}
The name ``double exponential'' in \cref{de} originates from the asymptotic behavior $(e^{\pm 2e^{\pm y}}-1)/(e^{\pm 2e^{\pm y}}+1)$ of \cref{fde} for large $y$, which converges to $\pm 1$ at a double-exponential rate as $y \to \pm \infty$. This is, of course, not the only variable transform with such behavior. For example, another transform bearing the same name is
\begin{align}
\psi(y) = \frac{e^{\pi \sinh (y)}}{1+e^{\pi \sinh (y)}},
\end{align}
which is used in investigating the resolution power of transplanted Chebyshev interpolation \cite{ric2}.

Alongside double-exponential transforms, there exist other exponential transforms that can also be applied in principle, including the single-exponential transform \cite{tak2,ric1,ric2}
\begin{align*}
\psi(y) = \tanh(y) ~~~~\text{or}~~~~ \psi(y) = \frac{e^y}{1+e^y},
\end{align*}
the IMT transform \cite{iri}
\begin{align}
\psi(y) = \displaystyle \left. \int_0^y e^{-\left(\frac{1}{s}+\frac{1}{1-s} \right)}\md s \middle/ \int_0^1 e^{-\left(\frac{1}{s}+\frac{1}{1-s}\right)}\md s \right.,
\end{align}
and the erf transform \cite{tak2}
\begin{align}
\psi(y) = \frac{2}{\sqrt{\pi}}\int_0^y e^{-s^2}\md s,
\end{align}
to name a few. When one of these transforms is incorporated into a given function, the transplanted function converges exponentially to constant values. Consequently, the transplanted function can be truncated to a finite interval, on which a polynomial-based interpolant can be constructed to approximate it.

We choose the double-exponential transform over others for two reasons. First, although a transplanted Chebyshev series based on any of these transforms approximates a given function at an exponential rate\footnote{Meanwhile, it is well known that the double-exponential, single-exponential, IMT, and erf transforms yield convergence rates of $\mO(e^{-n/\log n})$, $\mO(e^{-\sqrt{n}})$, $\mO(e^{-\sqrt{n}})$, and $\mO(e^{-n^{2/3}})$, respectively. One may wonder whether these rates contradict the statement of exponential convergence. This is not the case. These convergence rates arise from a delicate balance between the domain error (also known as truncation or endpoint error) and the interpolation error (also known as discretization or interior error). In our setting, since we work on the fixed domain $[-1, 1]$ with a fixed parameter $\omega$, the convergence rate for both single and double-exponential transforms is exponential. This is the same situation encountered in the construction phase of the \textsc{Sincfun} system \cite[\S 3]{ric1}.}, i.e., the error decays as $\mO(C^{-N})$, the constant $C$ is significantly larger for the double-exponential transform than for other transforms. In other words, for the double-exponential transform, $K$ and $L$ are much smaller than those for other mappings. Second, for the same $\mu$, the rank $r$ of the bivariate Chebyshev expansion of $G(y, t)$ is also significantly smaller. Hence, the double-exponential transform leads to a much faster construction and a lower-banded system with a much smaller bandwidth.

In principle, we can also employ one-sided single-exponential transforms, e.g., $\psi(y) = e^y$ for $y \in [-\infty, 0]$, or one-sided double-exponential transforms, e.g., $\psi(y) = e^{1-e^{-y}}$ for $y \in [-\infty, 0]$, in certain circumstances for reduced cost of approximation. However, if oscillatory behavior occurs at the endpoint other than the one where the singularity is located, it may incur even more cost than necessary. One such example is the fractional Airy equation (\cref{sec:airy}).

In \cite{adc1,adc2}, Adcock and his coauthors proposed parameterized slit-strip single- and double-exponential transforms that possess both optimal asymptotic convergence rates and resolution power. Slevinsky and Olver also proposed a modified double-exponential transform \cite{sle} that yields an improved convergence rate for approximating the transplanted functions. However, these transforms are all function-dependent and are therefore not applicable in the current context.

\section{Numerical examples}\label{sec:example}
We now demonstrate the effectiveness and accuracy of the variable-transform-based spectral approximation to FIOs by focusing on the double-exponential transform \cref{de} and on the solution of various fractional calculus problems. Extensive numerical experiments for the method based on the algebraic transform can be found in \cite{liu}. Most of the problems in this section that we solve by the DESA-based spectral method are previously intractable for existing spectral methods. 

We choose $K=L=100$ and $r = 30$ for the examples in this section, only except for the first example, where they are increased accordingly to deal with the radically small $\mu$. 

%

\subsection{Fractional Abel integral equation}\label{sec:abel}
Our warm-up problem is the simple fractional Abel integral equation
\begin{align}
u(x) + {}_{-1}\mathcal{I}_{x}^{\mu}[u](x) = f(x), \label{abel}
\end{align}
where the right-hand side
\begin{align*}
f(x) = \left(1+x\right)^{\mu} + \frac{\Gamma(1+\mu)}{\Gamma(1+2\mu)}\left(1+x\right)^{2\mu}.
\end{align*}
The exact solution is $u(x) = (1+x)^{\mu}$. What distinguishes this rather mundane example from similar numerical experiments in the literature is that we solve it for $\mu = 1, 10^{-1}, \dots, 10^{-6}$. As $\mu$ decreases, larger values of $r$, $K$, $L$, and $\omega$ are required to approximate $G(y, t)$ with sufficient accuracy. The values of $r$, $K$, and $L$ corresponding to each $\mu$ are listed in \cref{tab:K}. The values of $\omega$ are determined by \cref{beta3} and indicated by the vertical dashed lines in \cref{fig:beta}. 

The error in the computed solution for various truncation sizes $N$ is assessed by comparing with the exact solution at $10^4$ equispaced points in $[-1, 1]$, and is shown in \cref{fig:frac_left_eq}. After an initial phase of exponential decay to approximately machine precision, the errors level off without rebounce, which indicates the stability of the DESA-based spectral method. To the best of our knowledge, this appears to be the first numerical experiment in which fractional integral equations with such small fractional orders are solved to this level of accuracy, as existing methods can hardly go beyond $\mu = 0.1$.
\begin{table}[h!] 
\centering 
\caption{The values of $r$, $K$, and $L$ versus $\mu$ for approximating $G(y, t)$ to adequate accuracy.} \label{tab:K} \begin{tabular}{cccccccc} \hline \\[-9pt] $\mu$ & $1$ & $10^{-1}$ & $10^{-2}$ & $10^{-3}$ & $10^{-4}$ & $10^{-5}$ & $10^{-6}$ \\ \hline $r$ & 28 & 40 & 50 & 58 & 64 & 65 & 64 \\ $K = L$ & 80 & 100 & 120 & 170 & 350 & 660 & 920 \\ \hline 
\end{tabular} 
\end{table}

\begin{figure}[t!] 
\centering 
\subfloat[]{\label{fig:frac_left_eq} \includegraphics[width=0.45\linewidth]{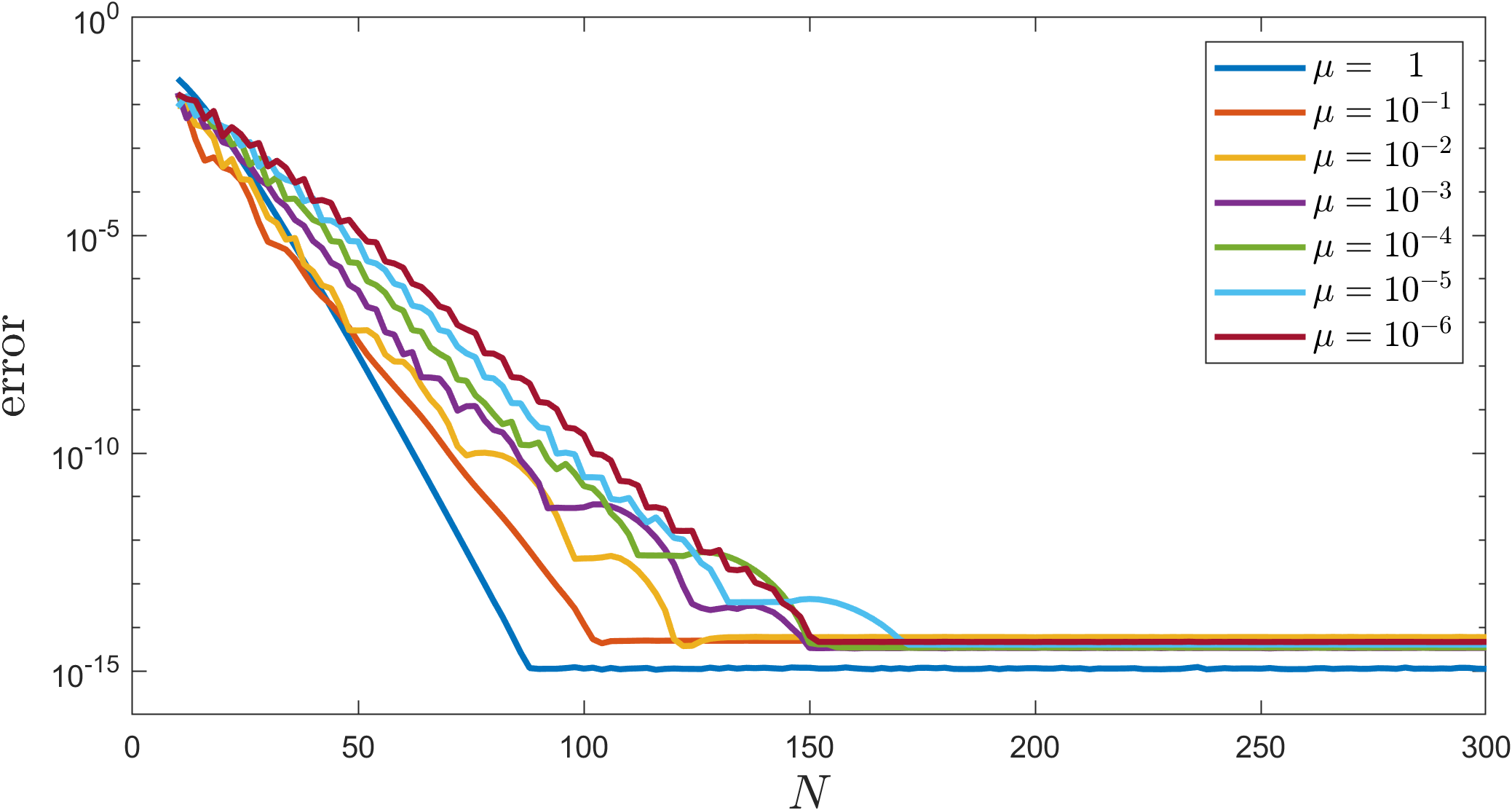} } \hfill \subfloat[]{\label{fig:riesz} \includegraphics[width=0.45\linewidth]{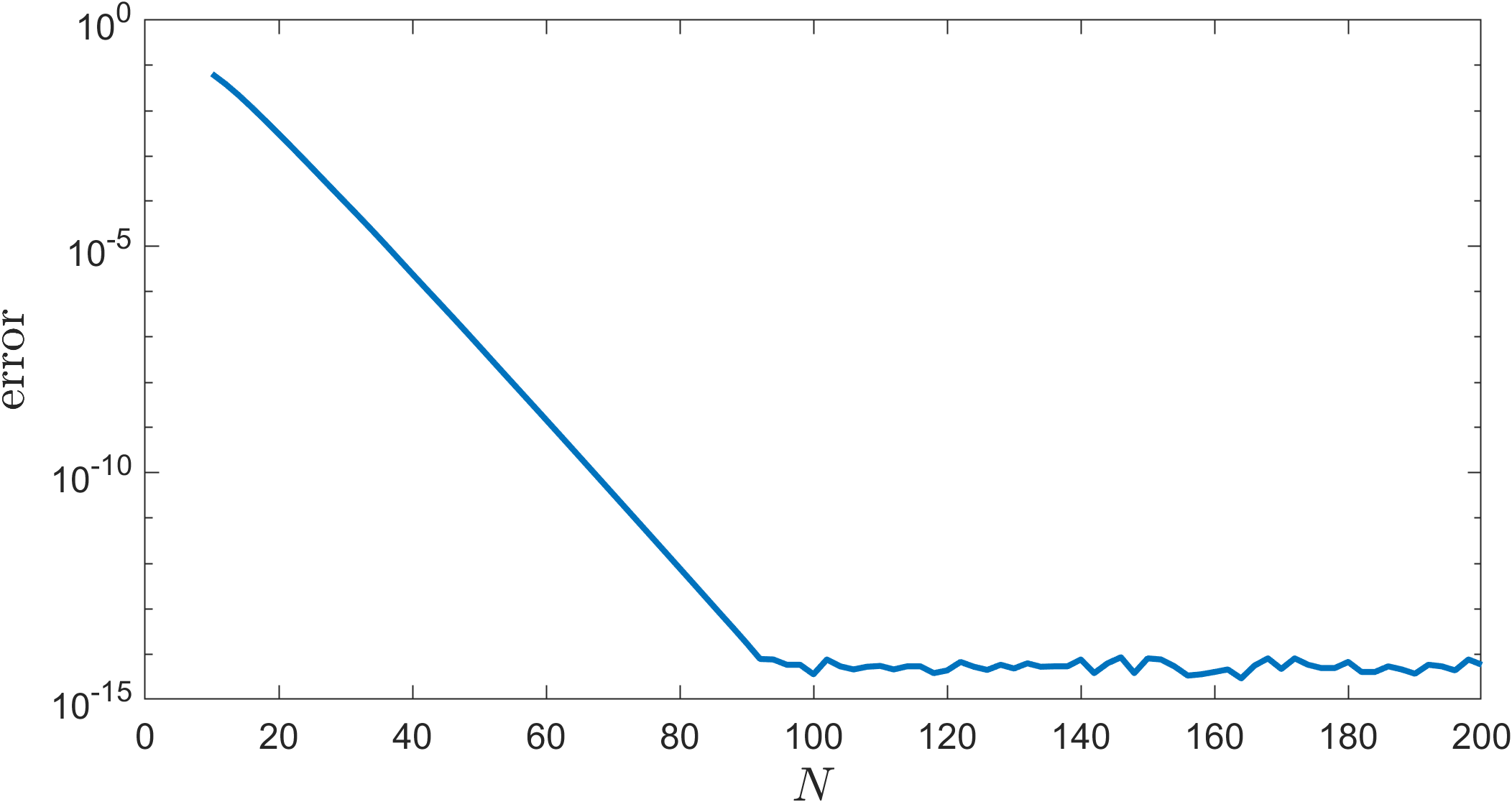} } 
\caption{Error versus $N$ for (a) the fractional Abel integral equation \cref{abel} of various fractional order $\mu$ and (b) the Riesz fractional integral equation \cref{ex:riesz}.} 
\end{figure}

\subsection{Riesz fractional integral equation}
The Riesz FIO of order $\mu$, where $\mu \neq 2k+1$ for $k \in \mathbb{N}$, is defined by
\begin{align*}
\mathcal{I}^{\mu}_R[f](x)
= \frac{1}{2\Gamma(\mu)\cos\left(\pi\mu/2\right)}
\int_{-1}^{1} \frac{f(t)}{\lvert x-t\rvert^{1-\mu}}\,\md t
= \frac{1}{2\cos\left(\pi\mu/2\right)}
\left({}_{-1}\mathcal{I}_{x}^{\mu} + {}_{x}\mathcal{I}_{1}^{\mu}\right)[f](x),
\end{align*}
which induces singularities at both endpoints $\pm 1$. To the best of our knowledge, no spectral method exists that is capable of solving FIEs involving Riesz FIOs. Consider the Riesz fractional integral equation 
\begin{align}
u(x) + \mathcal{I}_R^{1/2}[u](x) = f(x),\label{ex:riesz}
\end{align}
where the right-hand side function 
\begin{align*}
f(x) = 2\Gamma\left(1/2\right)\cos{\left(\pi/4\right)} \sqrt{1+x} + \sqrt{2(1-x)} + (1+x)\left(\frac{\pi}{2}+\operatorname{arctanh}\!\left(\sqrt{\frac{1-x}{2}}\right)\right)
\end{align*}
is chosen such that the exact solution is 
\begin{align*}
u(x) = 2\Gamma\left(1/2\right)\cos{\left(\pi/4\right)} \sqrt{1+x}.
\end{align*}

The error in the computed solution is plotted against $N$ in \cref{fig:riesz}, where the error decays exponentially to machine precision and then remains at that level.

\subsection{FIE with variable coefficients and multiple FIOs of different orders}
Our third example is
\begin{align}
u(x) + (1+x)^{2/3}\mathcal{I}_R^{\sqrt{2}}[u](x) + {}_{x}\mathcal{I}_{1}^{\pi/4}[(1-\rotatebox{45}{\scalebox{0.55}{$\square$}} )^{\sqrt{3}}u](x) + {}_{-1}\mathcal{I}_{x}^{e/3}[u](x) = 1, \label{mixed}
\end{align}
which features (1) FIOs of incompatible fractional orders, (2) singularities at both boundary points, and (3) fractional-power variable coefficients appearing as prefactors to both the FIOs and the solution. Obviously, this FIE does not fall into the category of \cref{fie} and therefore cannot be solved by the JFP method or any other existing spectral methods. This FIE can be successfully solved using DESA, and the solution is shown in \cref{fig:mixed_sub2}. The convergence of the method is evidenced by the decay of the Cauchy error, as shown in \cref{fig:mixed_sub1}. The Cauchy error decays exponentially as $N$ increases until it reaches $\mO(10^{-15})$ and then plateaus. 

\begin{figure}[H]
\centering
\subfloat[]{\label{fig:mixed_sub2}
  \includegraphics[width=0.45\linewidth]{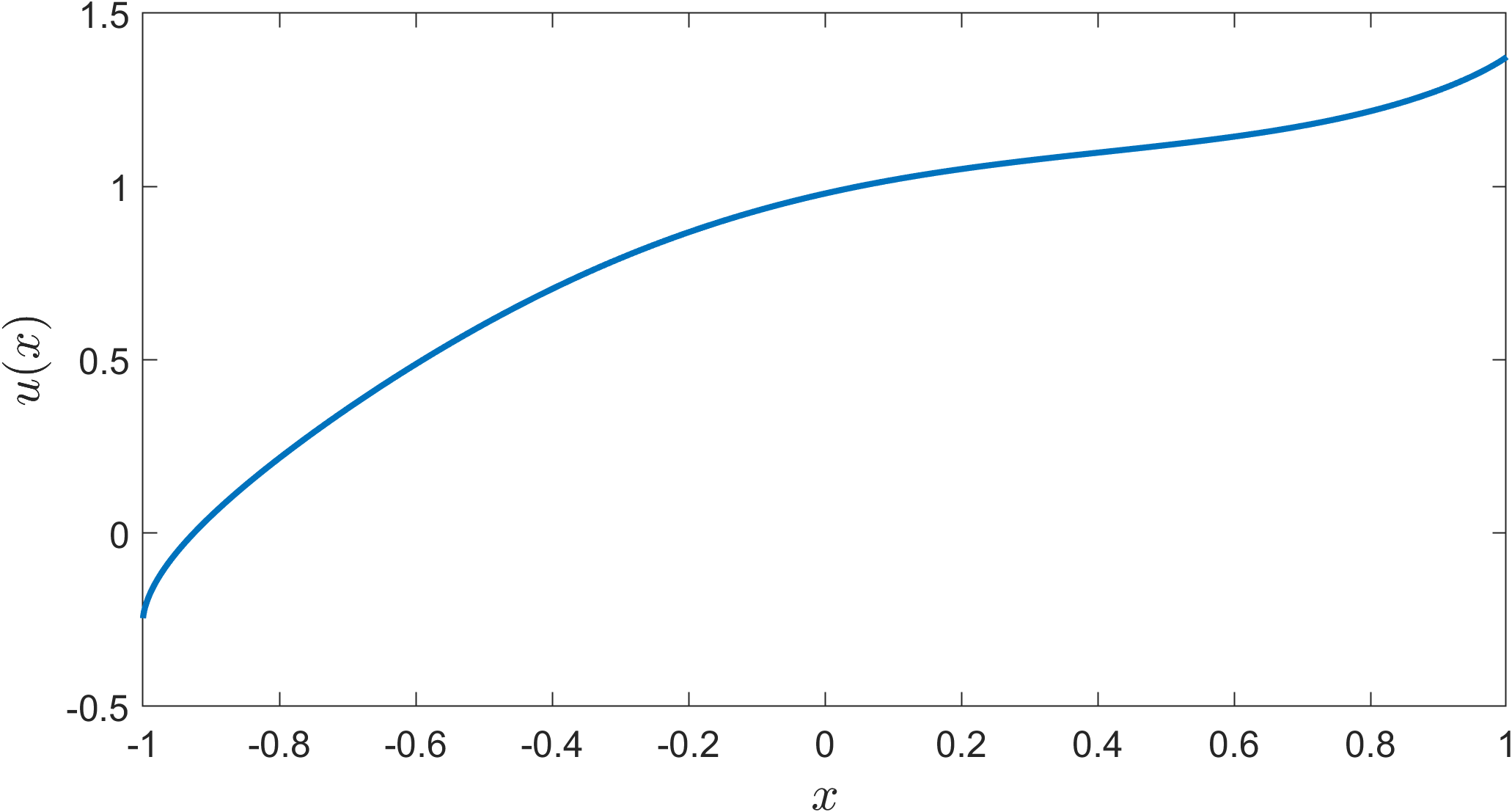}
}
\hfill
\subfloat[]{\label{fig:mixed_sub1}
  \includegraphics[width=0.455\linewidth]{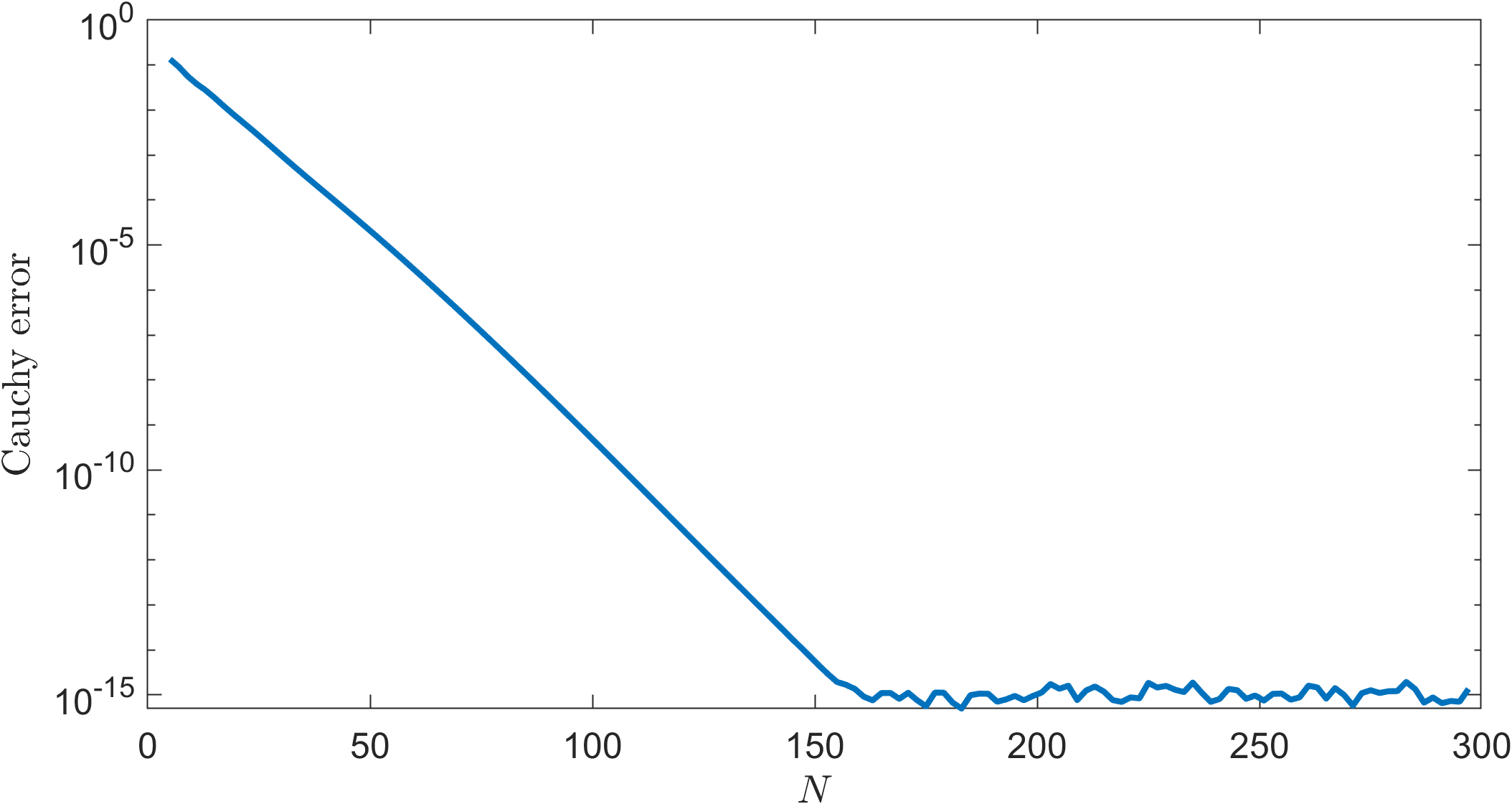}
}
\caption{Solving \cref{mixed} by the DESA-based spectral method: (a) the solution and (b) the Cauchy error.}
\end{figure}

\subsection{Fractional Airy equation}\label{sec:airy}
Our next example is the Airy equation of fractional differential order \cite{hal}
\begin{align}
\epsilon i^{3/2} {}^{RL}_{\kern0.1em-1}\mD_{x}^{3/2}[u](x) - xu(x) = 0, \quad x \in [-1,1], \quad \text{s.t. } u(-1) = 0,~~~ u(1) = 1. \label{airyeq}
\end{align}
For a singularly perturbed problem like \cref{airyeq}, the solution requires a very high-degree DETCP series approximation when $\epsilon$ is small. We take this example to demonstrate that the DESA-based spectral method remains stable when solving large-scale problems. Thus far, this problem can only be solved in a numerically stable manner by the JFP method \cite{liu}. Instabilities of different kinds have been observed when it is solved using the GLOFPG and SS spectral methods. 

To solve \cref{airyeq} via integral reformulation, we assume 
\begin{align}
u(x) = {}_{\kern0.3em-1}\mI_{x}^{3/2}[v](x) + a(1+x), \label{uv}
\end{align}
where $a$ is a constant to be determined. The ansatz in \cref{uv} ensures that the boundary condition at $x = -1$ is satisfied automatically. This reformulation converts the FDE into a linear system involving only FIOs. Substituting \cref{uv} into \cref{airyeq} yields
\begin{subequations}\label{airy}
\begin{align}
\epsilon i^{3/2}\left(\sqrt{1+x}\,v(x) + \frac{a}{\Gamma(1/2)}\right) - \left(x\sqrt{1+x}\right)\kern-0.4em {}_{\kern0.3em-1}\mI_{x}^{3/2}[v](x) - ax(1+x)^{3/2} = 0,\label{airyir}
\end{align}
subject to the boundary condition at the right endpoint
\begin{align}
{}_{\kern0.3em-1}\mI_{x}^{3/2}[v](1) + 2a = 1. \label{airyib}
\end{align}
\end{subequations}
\begin{figure}
\centering
\subfloat[]{\label{fig:airy1e-6}
  \includegraphics[width=0.45\linewidth]{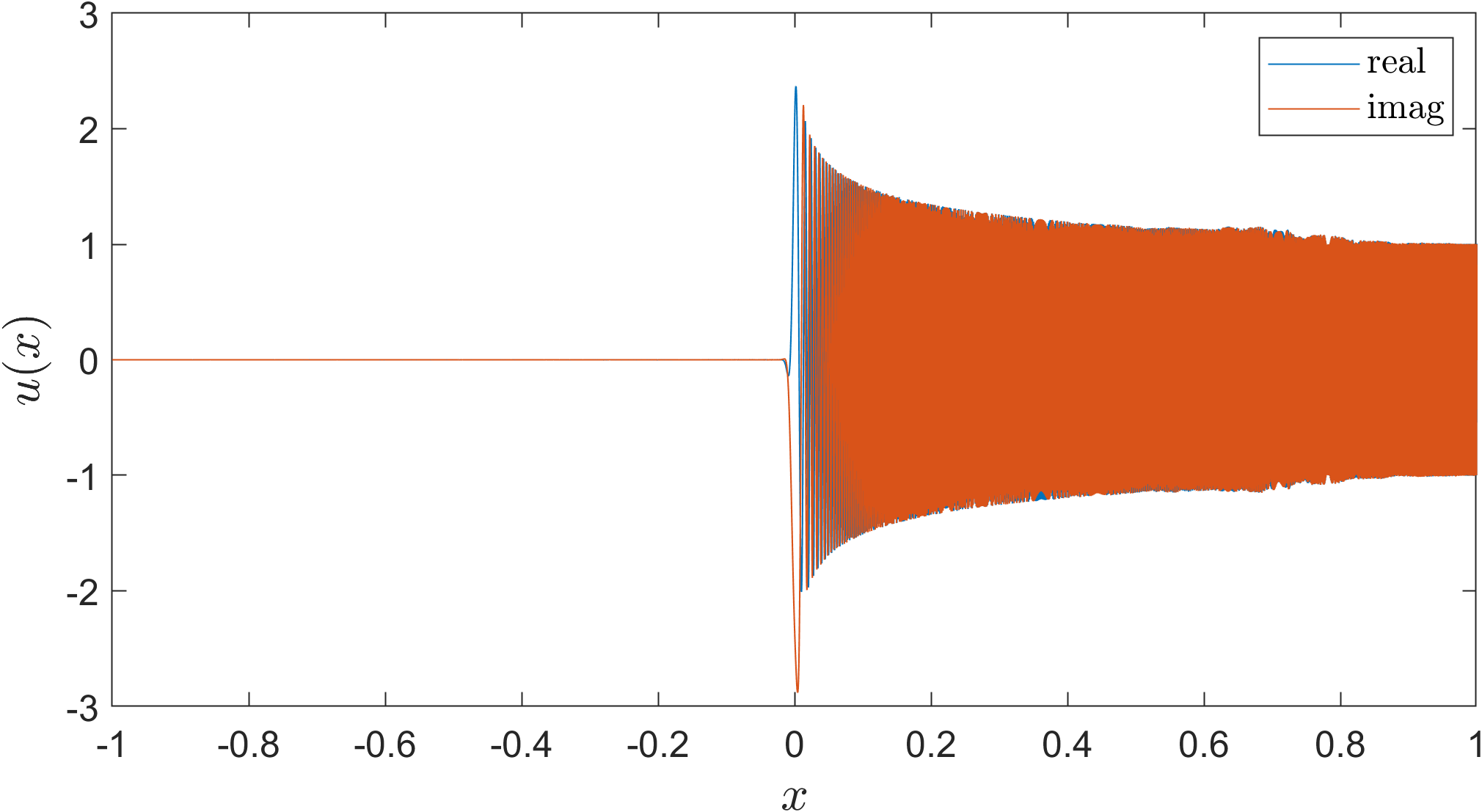}
}
\hfill
\subfloat[]{\label{fig:airy1e-6_cauchy}
  \includegraphics[width=0.47\linewidth]{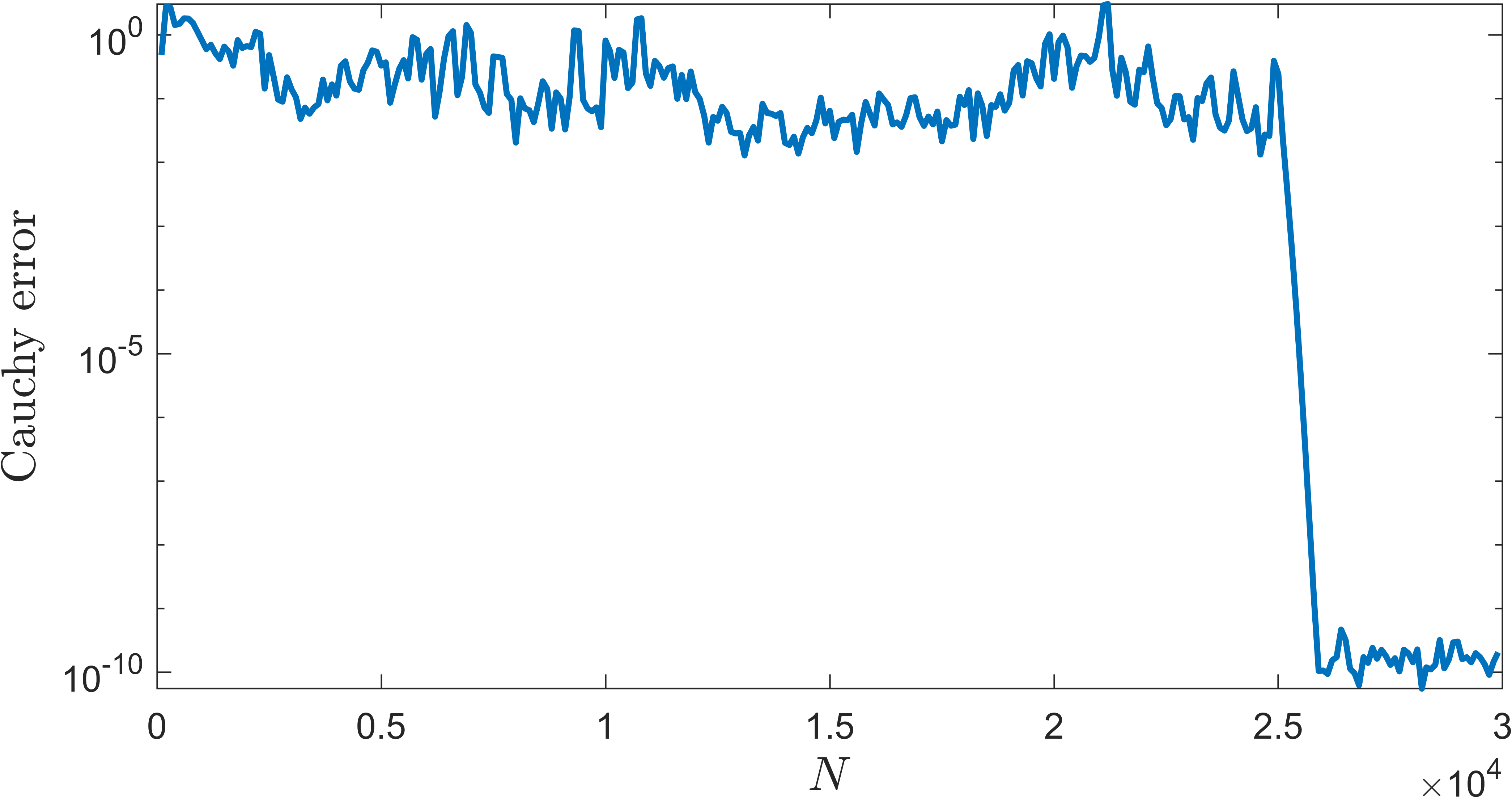}
}
\caption{Real and imaginary parts of the solution to the fractional Airy equation \cref{airy}, together with the corresponding Cauchy errors for $\epsilon = 10^{-6}$.}
\end{figure}

Let $\tilde{\mA}$ be the spectral approximation to ${}_{-1}\mI_{x}^{3/2}$, and let $\hat{v}$ be the coefficient vector of $v(x)$ in the DETCP basis. Replacing ${}_{-1}\mI_{x}^{3/2}$ and $v(x)$ in \cref{airy} with $\tilde{\mA}$ and $\hat{v}$ gives the following linear system:
\begin{align*}
\left(
\begin{array}{c|c}
\mB \tilde{\mA} & 2 \\[1mm]
\hline\\[-2mm]
\epsilon i^{3/2}\mM_1 - \mM_2 \tilde{\mA} & \hat{g} \\
\end{array}
\right)
\begin{pmatrix}
\hat{v}\\[2mm]
a
\end{pmatrix}
=
\begin{pmatrix}
1\\[2mm]
0
\end{pmatrix},
\end{align*}
where $\mB = (1, 1, \dots)$ encodes the boundary condition, $\mM_1$ and $\mM_2$ are the multiplication matrices corresponding to $\sqrt{1+x}$ and $x\sqrt{1+x}$ as given by \cref{thm:M}, and $\hat{g}$ is the vector of DETCP coefficients of 
\begin{align*}
g(x) = \frac{\epsilon i^{3/2}}{\Gamma(1/2)} - x(1+x)^{3/2}.
\end{align*}

\cref{fig:airy1e-6} displays the real and imaginary parts of the solution to \cref{airyeq} for $\epsilon = 10^{-6}$, computed with a truncation size $N = 26{,}500$. As shown in \cref{fig:airy1e-6_cauchy}, the Cauchy error decays geometrically. The sudden drop corresponds to the truncation size at which the solution begins to be resolved with a sufficient number of points per wavelength.

\subsection{Fractional eigenvalue problem}
The next example is the fractional eigenvalue problem
\begin{align}
\begin{array}{l}
{}^{RL}_{\kern0.1em-1}\mD_{x}^{\mu_1} [u](x) = -\lambda u(x) ~\text{ for }~ x \in [-1,1]\\[2mm]
\text{s.t. }~ u^{(j)}(-1) = 0 \text{ for } j = 0,1,\ldots,\ell-2 ~\text{ and }~ {}^{RL}_{\kern0.1em-1}\mD_{x}^{\mu_2}u(1) = 0,
\end{array}\label{eigdif}
\end{align}
where $\ell \geq 2$ is an integer and $\mu_1$ and $\mu_2$ satisfy $0 \leq \mu_2 < \ell - 1 < \mu_1 < \ell$. This eigenvalue problem plays a pivotal role in studying the properties of the two-parameter Mittag--Leffler function \cite{elo}. Here, $u^{(j)} = \md^j u(x)/\md x^j$. While this problem was solved in \cite{liu} by the JFP spectral method for $\mu_1 = 3/2$ and $\mu_2 = 0$, it becomes intractable for the JFP method or any other spectral methods when $\mu_1-1$, $\mu_1-\mu_2$, and $\mu_1$ do not satisfy the compatibility conditions stated below \cref{fie}. In contrast, the DESA-based spectral method can readily handle any admissible values of $\mu_1$ and $\mu_2$---we set $\mu_1 = 1.23456789$ and $\mu_2 = 0.123456789$ in our experiment.

\begin{figure}[t!]
\centering
\subfloat[]{\includegraphics[width=0.46\linewidth]{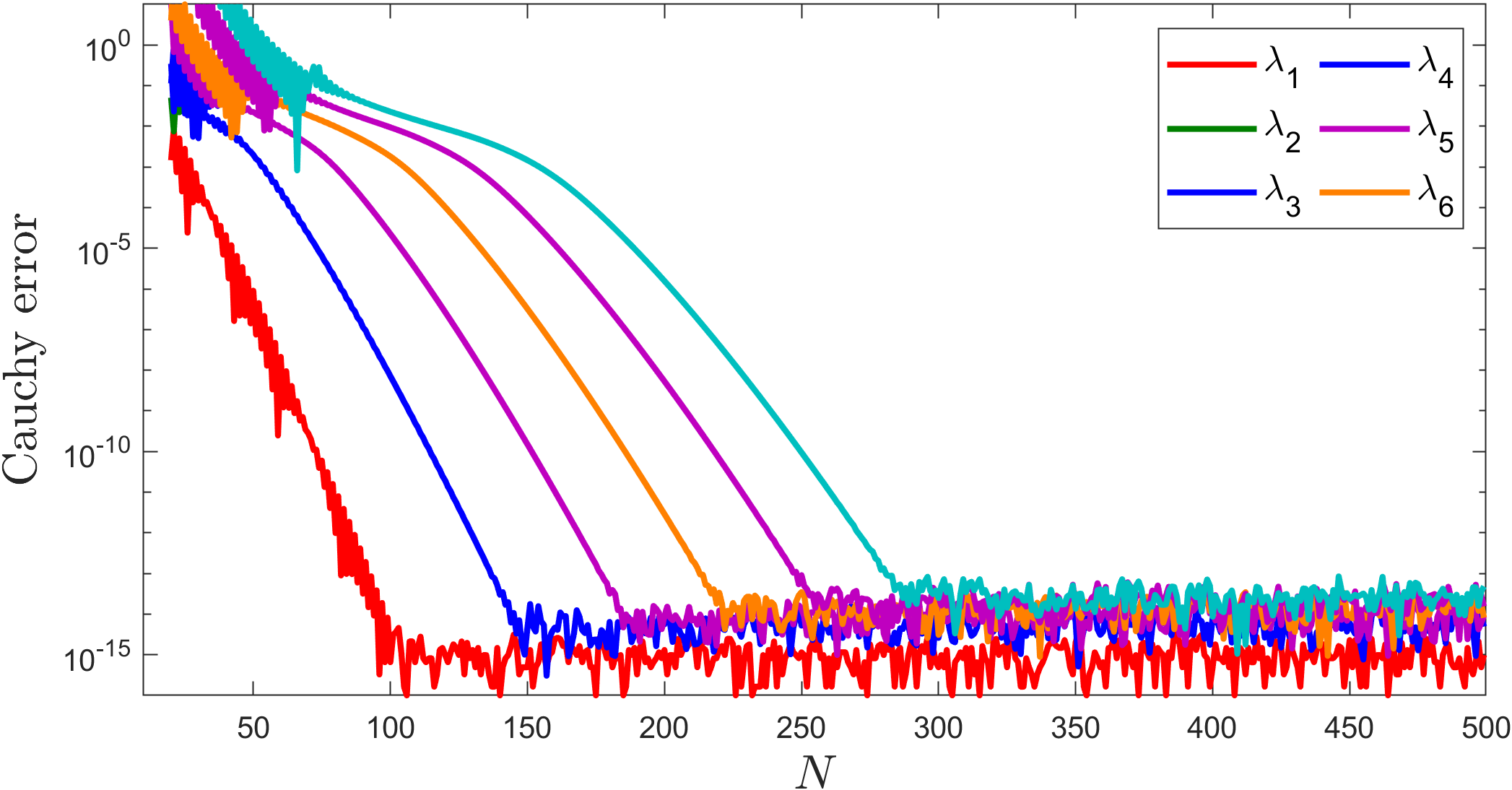}\label{fig:eig_cauchy}}
\hfill
\subfloat[]{\includegraphics[width=0.46\linewidth]{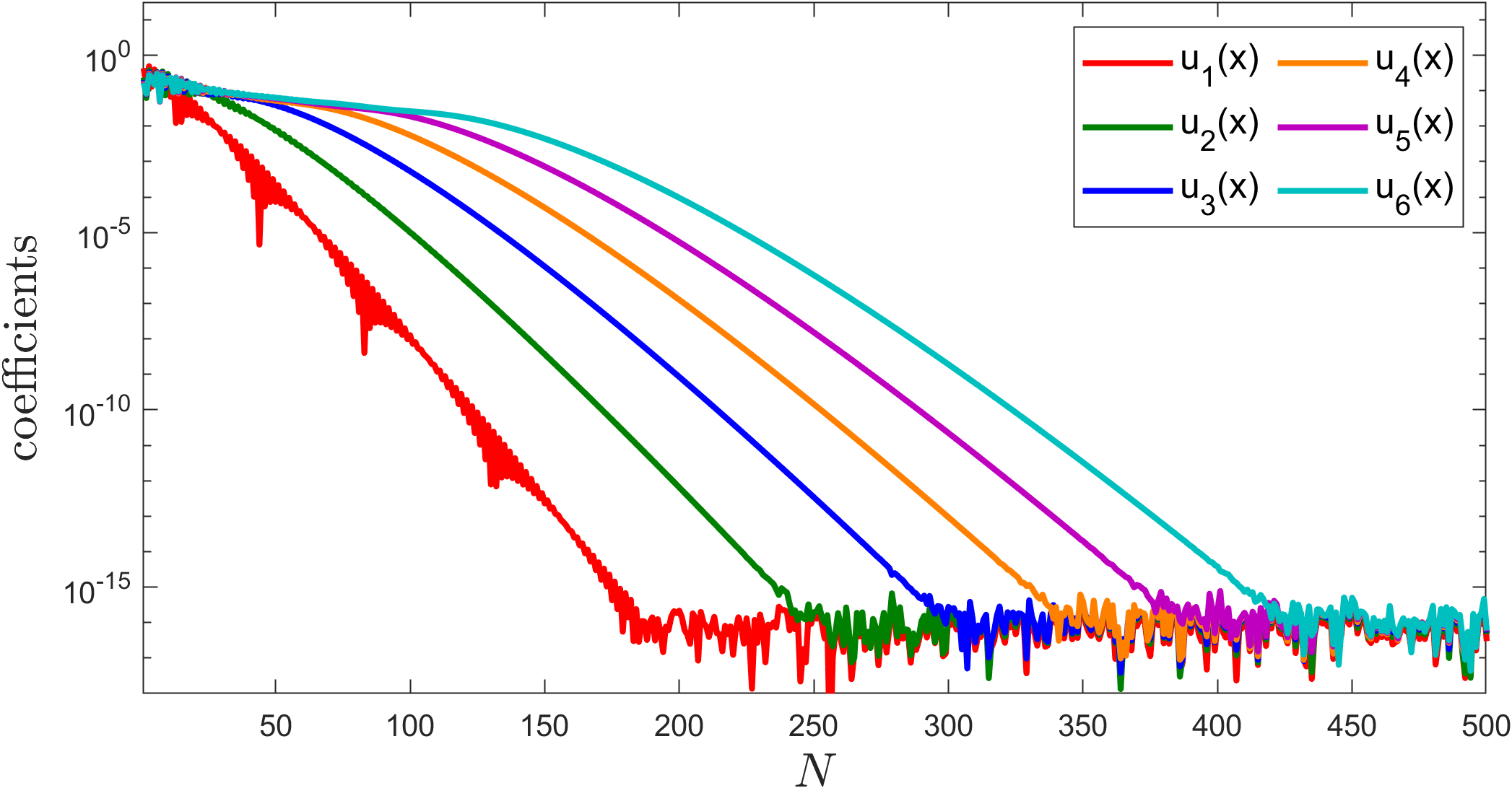}\label{fig:eigfun_coef}}
\hfill
\subfloat[]{\includegraphics[width=0.45\linewidth]{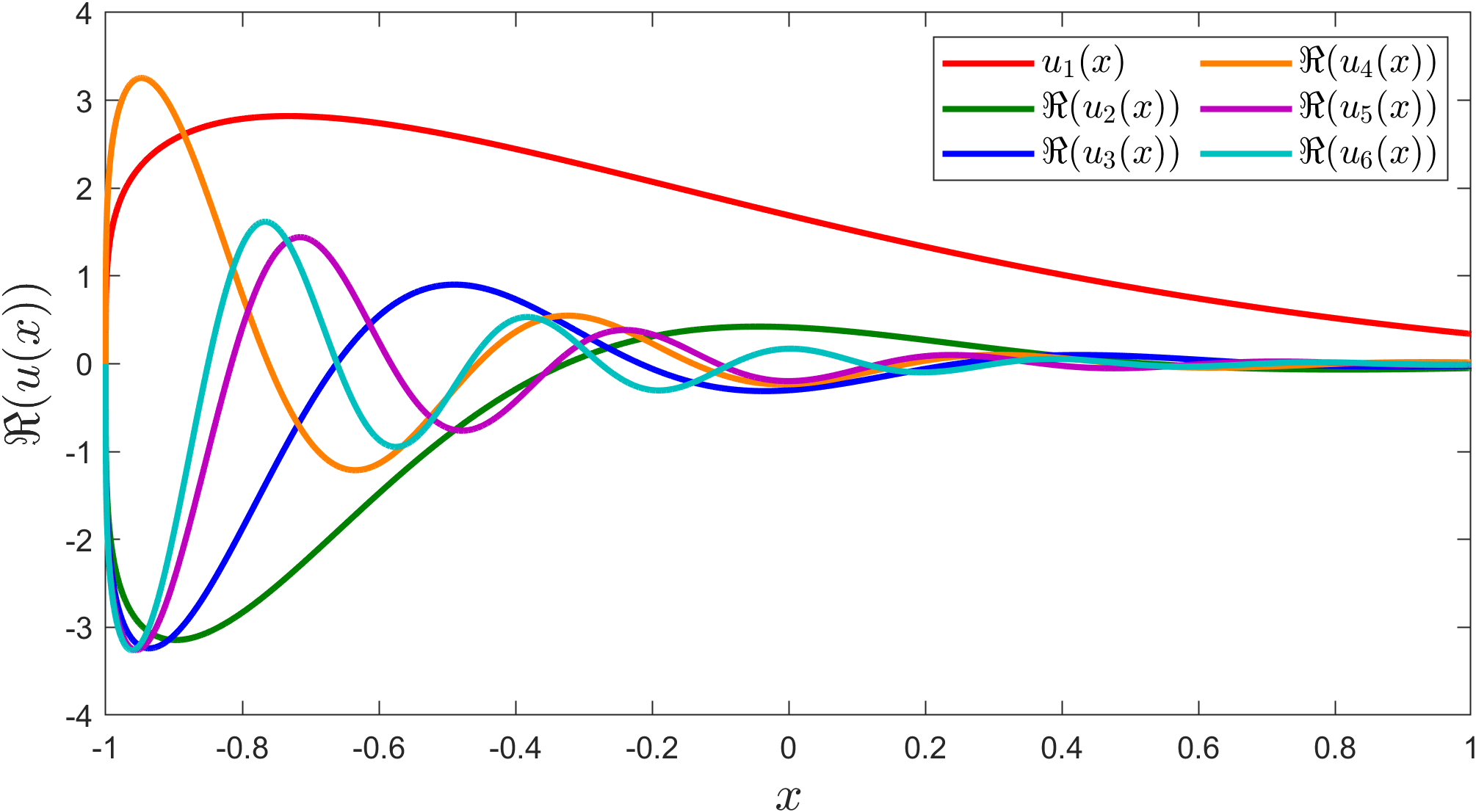}\label{fig:eigfun_real}}
\hfill
\subfloat[]{\includegraphics[width=0.46\linewidth]{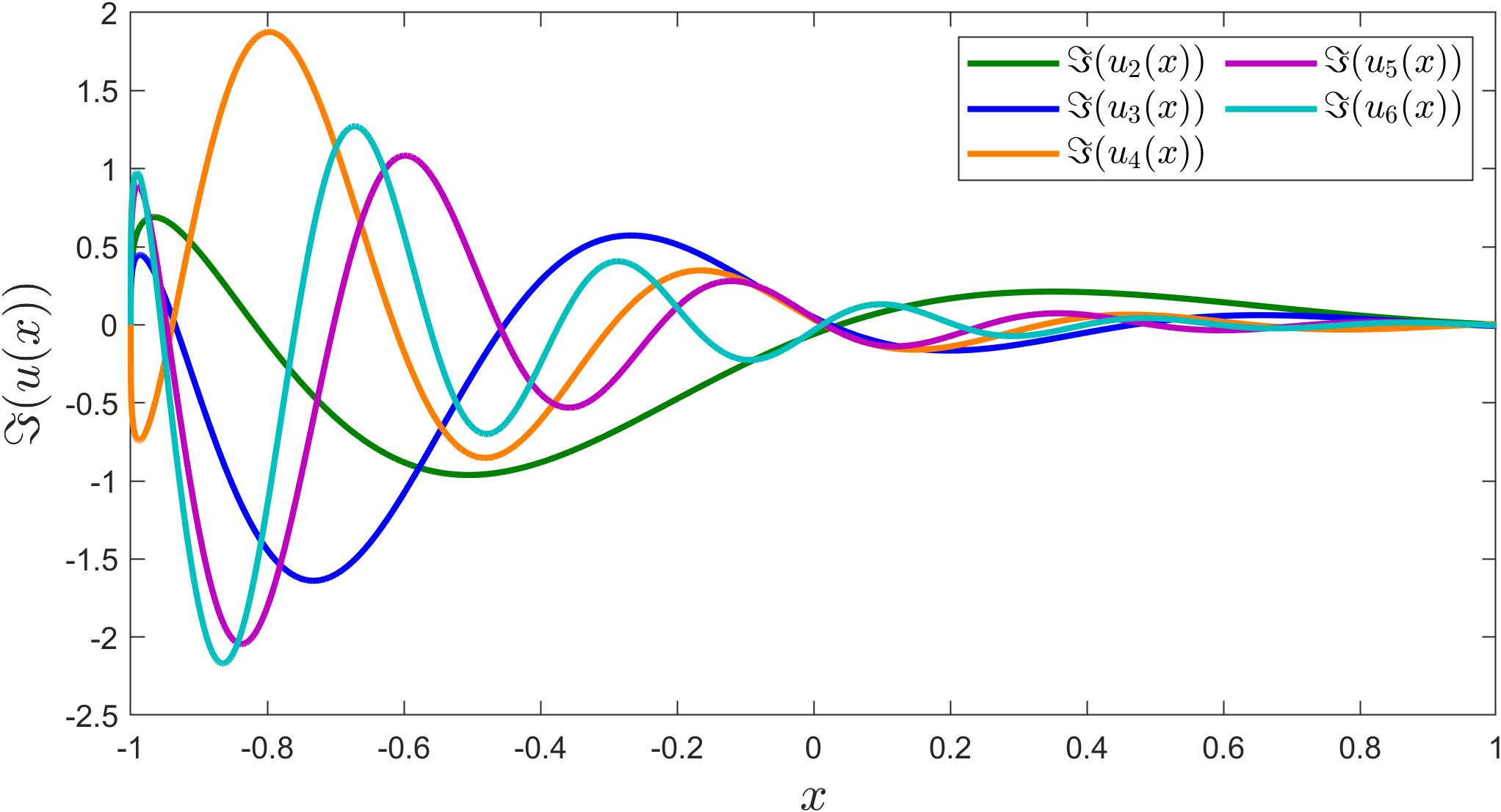}\label{fig:eigfun_imag}}
\caption{(a) Cauchy errors of the six eigenvalues of smallest modulus, obtained by solving the eigenproblem of consecutive truncation sizes using \texttt{eigs}. (b) DETCP coefficients of the eigenfunctions. The real (c) and imaginary (d) parts of the eigenfunctions corresponding to these six eigenvalues.}
\end{figure} 

Similar to FDEs, this problem can be reformulated as an FIO eigenvalue problem via integration:
\begin{align}
\vartheta (1+x)^{\mu_1 - 1} {}_{\kern-0.3em-1}\mI_{x}^{\mu_1 - \mu_2}[u](1) - {}_{\kern-0.3em-1}\mI_{x}^{\mu_1}[u](x) = \frac{1}{\lambda}u(x) , \label{eigint}
\end{align}
where $\vartheta = \displaystyle \frac{\Gamma(\mu_1-\mu_2)}{\Gamma(\mu_1)}2^{1+\mu_2-\mu_1}$ is a constant. Let $\hat{v}$ and $\hat{u}$ denote the infinite DETCP coefficient vectors of $(1+x)^{\mu_1 - 1}$ and an eigenfunction $u(x)$, respectively. Then the operator eigenvalue problem \cref{eigint} can be expressed as
\begin{align}
\left( \vartheta \hat{v} \mB \mA_1 - \mA_2 \right) \hat{u} = \varrho \hat{u}, \label{eigintmat}
\end{align}
where $\varrho = 1/\lambda$, $\mB = (1, 1, \ldots)$, and $\mA_1$ and $\mA_2$ are the spectral approximations to ${}_{\kern-0.3em-1}\mI_{x}^{\mu_1 - \mu_2}$ and ${}_{\kern-0.3em-1}\mI_{x}^{\mu_1}$, respectively.

To approximate the six eigenvalues of smallest modulus and their corresponding eigenfunctions of \cref{eigdif}, we truncate \cref{eigintmat} and compute the six eigenpairs corresponding to the six values of $\varrho$ of largest modulus by calling \textsc{Matlab}'s \texttt{eigs}. This procedure is repeated with progressively larger truncation sizes. The Cauchy error of the computed eigenvalues is measured by the $2$-norm of the difference between eigenvalues obtained at two consecutive truncation sizes. Both the Cauchy errors and the computed eigenvectors, which approximate the DETCP coefficients of the eigenfunctions, are examined for plateaus \cite{aur}. The iterative process is terminated only when both quantities exhibit plateaus (see \cref{fig:eig_cauchy,fig:eigfun_coef}), which serve as indicators of convergence and ensure accurate approximations to the eigenvalues and eigenfunctions \cite[Chap.~IV, \S 3.5]{kat}. The corresponding eigenfunctions are shown in \cref{fig:eigfun_real,fig:eigfun_imag} for the real and imaginary parts.

We list the computed values of the six eigenvalues in \cref{tab:valE}, which, except for the first one, are complex. Since $\lambda$ is an eigenvalue of \cref{eigdif} if and only if it is a zero of $E_{\mu_1, \mu_1-\mu_2}(-2^{\mu_1}\lambda)$, we also report the value of $E_{\mu_1,\mu_1-\mu_2}(-2^{\mu_1}\lambda)$ for each computed $\lambda$. The results indicate that the computed eigenvalues are accurate up to machine precision.

\begin{table}[t!]
\centering
\renewcommand{\arraystretch}{1.4} 
\setlength{\tabcolsep}{7pt}
\caption{The six eigenvalues of smallest modulus, along with the corresponding values of $E_{\mu_1,\mu_1-\mu_2}(-2^{\mu_1}\lambda)$.}\label{tab:valE}
\begin{tabular}{c c c}
\toprule
index & eigenvalue $\lambda$ & $\left\lvert E_{\mu_1,\mu_1-\mu_2}(-2^{\mu_1}\lambda)\right\rvert $ \\
\midrule
1 & $1.355201481588489$ & $1.20\times 10^{-14}$ \\
\midrule
2 & $4.930015412112804 \pm 3.094646626469975i$ & $2.36\times 10^{-14}$ \\
\midrule
3 & $8.217665634311189 \pm 8.089668419364024i$ & $1.05\times 10^{-13}$ \\
\midrule
4 & $11.387725243090559 \pm 13.804866459545323i$ & $2.27\times 10^{-13}$ \\
\midrule
5 & $14.564020446706387 \pm 20.023206234983594i$ & $9.04\times 10^{-14}$ \\
\midrule
6 & $17.778713260368924 \pm 26.640854355716016i$ & $7.87\times 10^{-13}$ \\
\bottomrule
\end{tabular}
\end{table}

\subsection{Pseudospectra of an FDO}
Our final example is the pseudospectra of the left-sided FDO of Caputo type
\begin{align}
{}_0^{\kern-0.2em C}\mD_x^{1/2}[f](x) = \frac{1}{\Gamma(1/2)}\int_0^x \frac{f'(t)}{(x-t)^{1/2}} \md t, \label{caputo}
\end{align}
subject to the boundary condition $f(0)=0$. By definition, the $\veps$-pseudospectrum of ${}_0^{\kern-0.2em C}\mD_x^{1/2}$ is the set of complex numbers $z$ such that 
\begin{align*}
\sigma_{\veps}=\left\{ z \in \mathbb{C} : \|\mR(z)\| > \veps^{-1} \right\}
\end{align*}
for $\veps > 0$. Here, $\mR(z) = \left(z\mJ - {}_0^{\kern-0.2em C}\mD_x^{1/2}\right)^{-1}$ and $\mJ$ are the resolvent and the identity operator, respectively. This problem is closely related to the famous Hille--Phillips problem \cite{hil} and is studied in \cite[\S 19]{tre}. For this operator, the spectrum is empty, i.e., \cref{caputo} has no eigenvalues at all, and the resolvent norm is exponentially large in a quadrant of the complex plane.

Trefethen and Embree computed its $\veps$-pseudospectra by a spectral collocation method following a ``discretize-then-solve'' strategy. Since the ``discretize-then-solve'' approach is at high risk of \emph{spectral pollution} and \emph{spectral invisibility}, it is more sensible and safer to compute the pseudospectra using the ``solve-then-discretize'' paradigm proposed in \cite{den}. The method is based on the key observation that
\begin{align}
\|\mR(z)\|^2 = \|\mR^*(z)\mR(z)\| = \lambda_{\max}(\mR^*(z)\mR(z)),
\end{align}
where $\lambda_{\max}(\cdot)$ denotes the largest eigenvalue. The adjoint operator \cite{agr} is given by
\begin{align*}
\mR^*(z) = \left(\bar{z}\kern0.15em\mJ - {}^{RL}_{\kern0.6em x }\mD_1^{1/2}\right)^{-1},
\end{align*}
where the right-sided Riemann--Liouville FDO
\begin{align*}
{}^{RL}_{\kern0.6em x }\mD_1^{1/2}[f](x) = -\frac{1}{\Gamma(1/2)}\frac{\md}{\md x}\int_{x}^1 \frac{f(t)}{(t-x)^{1/2}} \md t ~\text{ s.t. }~f(1)=0.
\end{align*}
Hence, we set up a Cartesian grid in the region of interest $[-2, 12]\times [-7,7]$ in the complex plane and calculate $\lambda_{\max}(\mR^*(z)\mR(z))$ by the Lanczos method at each grid point $z$. Suppose that $u(x)$ is the iterate. In each Lanczos iteration we apply $\mR^*(z)\mR(z)$ to $u(x)$, i.e.,
\begin{align*}
w(x) = \mR^*(z)\mR(z)[u](x) = \left(\bar{z}\kern0.15em\mJ - {}^{RL}_{\kern0.6em x }\mD_1^{1/2}\right)^{-1}\left(z\mJ - {}_0^{\kern-0.2em C}\mD_x^{1/2}\right)^{-1}[u](x).
\end{align*}
It amounts to successively solving two FDEs
\begin{subequations}
\begin{align}
\left(z\mJ - {}_0^{\kern-0.2em C}\mD_x^{1/2}\right) [v](x) = u(x)~\text{ s.t. }~v(0)=0, \label{lanczos1} \\
\left(\bar{z} \kern0.15em \mJ - {}^{RL}_{\kern0.6em x }\mD_1^{1/2}\right)[w](x) = v(x)~\text{ s.t. }~w(1)=0. \label{lanczos2}
\end{align}%
\end{subequations}
The value $1/\sqrt{\lambda_{\max}\left(\mR^*(z)\mR(z)\right)}$ at each $z$ is stored upon convergence of the Lanczos iteration. Once this value is available over the entire grid, it is fed into a contour plotter for a graphical display of the $\veps$-pseudospectra. 

Existing spectral methods are unable to solve \cref{lanczos1,lanczos2} in succession using exactly the same basis functions. For example, the JFP method requires different algebraic transforms for left- and right-sided singularities. With the double-sided double-exponential transform, both singularities can be handled with exactly the same basis functions. Hence, we apply ${}_{0}\mI_{x}^{1/2}$ to both sides of \cref{lanczos1} to reformulate it as an FIE
\begin{align*}
\left(z\kern0.15em {}_0\mI_{x}^{1/2} - \mJ\right)[v](x) = {}_0\mI_{x}^{1/2}[u](x),
\end{align*}
and apply ${}_{x}\mI_{1}^{1/2}$ to both sides of \cref{lanczos2} to obtain
\begin{align*} 
\left(\bar{z}\kern0.15em {}_{x}\mI_{1}^{1/2} - \mJ\right)[w](x) = {}_{x}\mI_{1}^{1/2}[v](x).
\end{align*}
The rest of the computation follows the same procedure as for the fractional Airy equation. 

The $\veps$-pseudospectra obtained by the ``solve-then-discretize'' approach are shown in \cref{fig:pseudospectra}, and they are visually indistinguishable from those given in Figure 19.2 of \cite[\S 19]{tre}.

\begin{figure}[t!]
\centering
\includegraphics[width=0.5\linewidth]{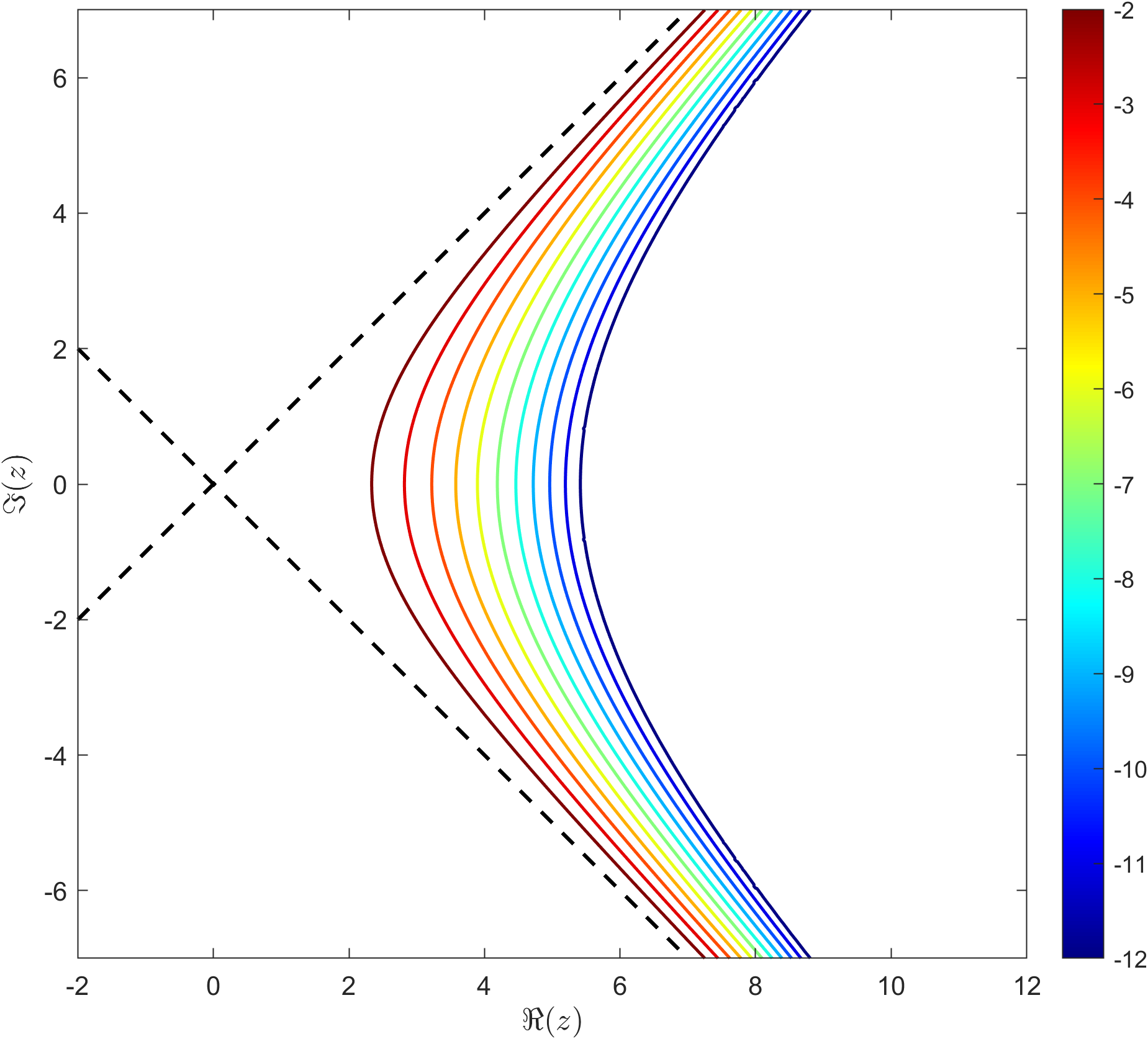}
\caption{$\veps$-pseudospectra of the half-derivative FDO \cref{caputo} for $\veps = 10^{-2}$, $10^{-3}, \ldots, 10^{-12}$, from left to right.}\label{fig:pseudospectra}
\end{figure}


\section{Discussion}\label{sec:remarks}

The basis function \cref{glof} used by the GLOFPG method is an algebraically weighted composition of the generalized Laguerre polynomial of parameter $\delta$ and the mapping
\begin{align}
y = \psi^{-1}(x) = -(\nu+1)\log x, \label{glofmap}
\end{align}
which maps $(0, 1)$ to $(0, \infty)$. 
Since the mapping \cref{glofmap} is nothing other than the inverse of the single-exponential transform
\begin{align*}
x = \psi(y) = e^{-\frac{y}{\nu+1}},
\end{align*}
approximating a given function with \cref{glof} is equivalent to composing the given function with $\psi(y)$ and then approximating the transplanted function $f(\psi(y))$ by a series of weighted Laguerre polynomials. This explains the first two drawbacks of the GLOFPG method mentioned immediately after \cref{glof}. On the one hand, because the generalized Laguerre polynomials form an unbounded basis in $C^{\infty}(0, \infty)$, the linear system of the GLOFPG method is inevitably ill-conditioned. On the other hand, the limited range of floating-point numbers implies that underflow and overflow are unavoidable whenever numbers outside this range arise in the computation. 

Unlike the GLOFPG method, most variable-transform-based methods truncate the transplanted function $f(\psi(y))$ and restrict the polynomial approximation to a finite domain where the transplanted function is nontrivial. This is why the DESAs are free from ill-conditioning as well as overflow and underflow.

Since the $n$th Chebyshev polynomial is itself obtained by composing $\cos(n\theta)$ with $\theta = \arccos x$, a TCP can be regarded as a doubly composed function involving $\cos(n\theta)$, $\theta = \arccos x$, and a variable transform designed to handle endpoint singularities.

Whereas the rank of $G(y, t)$ is on the order of dozens or hundreds for the double-exponential transform, it is exactly one for the algebraic transform. This motivates us to interpret the DESA as a linear combination of multiple JFP spectral approximations to FIOs of various fractional orders. In other words, when an FIO is approximated using DETCPs as basis functions, it can be viewed as a decomposition into multiple FIOs, each of which is approximated using a different JFP basis. Equivalently, we may view a DETCP as a spectral synthesis of infinitely many JFPs, of which only finitely many are numerically significant in approximating an FIO. This viewpoint helps explain the versatility of the DESA-based spectral method.

Despite the success and versatility of the spectral method based on the double-exponential transform, we recommend using the JFP spectral method \cite{liu} whenever possible due to its superior speed.

\section{Outlook}\label{sec:out}
For the double-exponential transform, the computational complexity depends on the rank $r$ and the degrees $K$ and $L$, as shown in \cref{alg:de}. We wonder whether new exponential transforms can be designed for which the corresponding $r$, $K$, and $L$ are smaller than those for \cref{de}. Furthermore, extensive experiments show that highly accurate, i.e., close to machine precision, results can be obtained with $r$, $K$, and $L$ that are much smaller than those required for $G(y, t)$ to be approximated to machine precision by a bivariate Chebyshev expansion. The mechanism behind this phenomenon deserves a systematic investigation, which may lead to significant and reliable savings in computational cost. Moreover, we have observed that the entries of $\mA$ above a certain upper diagonal are very small, e.g., $\mO(10^{-10})$. Although these entries are not exactly zero, it is of interest to determine whether new exponential transforms can be designed to further reduce their magnitudes so that the spectral approximation to FIOs can effectively be banded---this would bring significant speedup in both the construction and application of the spectral approximations, particularly in solving FIEs and FDEs. Finally, the definition of $G(y, t)$ is \emph{not} essential. Other definitions may work equally well if the ultimate goal is achieved, i.e., separating $y$ and $t$ in the bivariate function $\bigl(\psi(y)-\psi(y-(1+y)t)\bigr)^{\mu}$ in \cref{psi0}. It would be encouraging if the separation of variables can be achieved in a different but more efficient and less costly way.

In \cite{liu} and this work, FDEs and FDO eigenvalue problems are solved via integration reformulation. It would be most convenient if spectral approximations to FDOs can be constructed using TCPs following the same pattern, so that FDEs can be solved directly without preliminary reformulation.

The success and versatility of the proposed framework encourage us to consider problems beyond one spatial dimension, e.g., the fractional Laplacian. Other exciting but nontrivial extensions of the current work include variable-transform-based spectral methods for variable- and distributed-order fractional calculus. 

Finally, the framework presented in this paper is applicable beyond FIOs and FDOs to any Volterra integral operator for which an analogous three-term recurrence relation can be established. Moreover, after applying the change of variable \cref{cov} to the indefinite integral in \cref{IQnl}, the main part of the problem is converted from a Volterra integral to a Fredholm one. Thus, the proposed framework also applies to integral operators of Fredholm type. In other words, spectral approximations to general integral operators can be constructed in the same way by establishing a three-term recurrence relation satisfied by the result of applying the integral operator to Chebyshev or other orthogonal polynomials. Furthermore, we wonder whether the proposed paradigm for approximating integral operators can be extended to more general linear operators. We leave these questions for future work.

The \textsc{Matlab} code accompanying this paper is available at \cite{code}.

\bibliographystyle{siamplain}
\bibliography{references}
\end{document}